\documentclass{amsart} %book,article,report,letter
\usepackage[english]{babel}
\usepackage{amssymb,enumerate,bbm,amsmath,amsthm,mathrsfs}
\usepackage[colorlinks=true,linkcolor=blue,citecolor=blue,urlcolor=blue]{hyper ref}
\numberwithin{equation}{section}

\newtheorem{theo}{Theorem}[section]
\newtheorem{pro}[theo]{Proposition}
\newtheorem{lem}[theo]{Lemma}
\newtheorem{cor}[theo]{Corollary}

\newtheorem{rem}[theo]{Remark}

\newtheorem*{thmA}{Theorem A}
\newtheorem*{thmB}{Theorem B}
\newtheorem*{thmC}{Theorem C}
\newtheorem*{thmD}{Theorem D}

\renewcommand{\(}{\left(}
\renewcommand{\)}{\right)}
\renewcommand{\~}{\tilde}
\renewcommand{\-}{\overline}

\newcommand{\R}{\mathbb{R}}

\renewcommand{\S}{\mathbb{S}}
\renewcommand{\H}{\mathbb{H}}

\renewcommand{\a}{\alpha}

\newcommand{\g}{\gamma}
\renewcommand{\d}{\delta}
\newcommand{\e}{\varepsilon}
\renewcommand{\k}{\kappa}
\renewcommand{\l}{\lambda}

\newcommand{\D}{\Delta}

\renewcommand{\t}{\theta}
\newcommand{\s}{\sigma}

\newcommand{\G}{\Gamma}

\newcommand{\T}{\Theta}
\renewcommand{\L}{\Lambda}

\newcommand{\ra}{\rightarrow}

\newcommand{\mrm}{\mathrm}

\newcommand{\Vol}{\mrm{Vol}}

\begin{document}
\title[Harmonic mean curvature flow and geometric inequalities]{Harmonic mean curvature flow and geometric inequalities}
\author{Ben Andrews, Yingxiang Hu, Haizhong Li}

\address{Mathematical Sciences Institute \\ Australian University \\ ACT 2601 \\ Australia }
\email{ben.andrews@anu.edu.au}
\address{Yau Mathematical Sciences Center \\ Tsinghua University \\ Beijing 100084 \\ China }
\email{huyingxiang@mail.tsinghua.edu.cn}
\address{Department of Mathematical Sciences \\ Tsinghua University \\ Beijing, 100084 \\ China }
\email{lihz@tsinghua.edu.cn}
\date{}
\thanks{}
\begin{abstract}
In this article, we will use the harmonic mean curvature flow to prove a new class of Alexandrov-Fenchel type inequalities for strictly convex hypersurfaces in hyperbolic space in terms of total curvature, which is the integral of Gaussian curvature on the hypersurface. We will also use the harmonic mean curvature flow to prove a new class of geometric inequalities for horospherically convex hypersurfaces in hyperbolic space. Using these new Alexandrov-Fenchel type inequalities and the inverse mean curvature flow, we obtain an Alexandrov-Fenchel inequality for strictly convex hypersurfaces in hyperbolic space, which was previously proved for horospherically convex hypersurfaces by Wang and Xia \cite{Wang-Xia2014}. Finally, we use the mean curvature flow to prove a new Heintze-Karcher type inequality for hypersurfaces with positive Ricci curvature in hyperbolic space.
\end{abstract}

{\maketitle}
\section{Introduction}
The Alexandrov-Fenchel type inequalities for hypersurfaces in space forms have been extensively investigated by many authors. Using the inverse curvature flow and the optimal Sobolev inequality of Beckner \cite{Beckner1992}, Wei, Xiong and the third author \cite{Li-Wei-Xiong2014} proved a geometric inequality for two-convex (i.e., $p_1>0$ and $p_2>0$) and starshaped hypersurfaces in hyperbolic space, where $p_k$ is the (normalized) $k$-th mean curvature.
\begin{thmA}[\cite{Li-Wei-Xiong2014}]
If $\Sigma$ is a $2$-convex and starshaped hypersurface in $\H^n$, then
\begin{align}\label{1.1}
\int_\Sigma p_2 \geq |\Sigma|+\omega_{n-1}^\frac{2}{n-1}|\Sigma|^\frac{n-3}{n-1},
\end{align}
where $\omega_{n-1}$ is the area of the unit sphere $\S^{n-1}\subset \R^n$ and $|\Sigma|$ is the area of the hypersurface $\Sigma$, respectively. Equality holds in (\ref{1.1}) if and only if $\Sigma$ is a geodesic sphere.
\end{thmA}

It is observed by Ge, Wang and Wu \cite{Ge-Wang-Wu2014} that (\ref{1.1}) can be rewritten as
\begin{align*}
\int_\Sigma (p_2-1)\geq \omega_{n-1}^\frac{2}{n-1}|\Sigma|^\frac{n-3}{n-1},
\end{align*}
and the integrand $p_2-1$ is a constant multiple of the first Gauss-Bonnet curvature $L_1$. Here $L_k$ is the $k$-th Gauss-Bonnet curvature for hypersurfaces in space form $\mathbb{M}^n(\e)$ with constant curvature $\e\in \{-1,0,1\}$, which can be expressed as
$$
L_k(\k)=C_{n-1}^{2k}(2k)!\sum_{i=0}^{k}C_k^i \e^{i} p_{2k-2i}(\k),
$$
see (\ref{2.9}) for details. By establishing the (non-trivial) monotonicity of the functional
$$
Q(t):=|\Sigma_t|^{-\frac{n-1-2k}{n-1}}\int_{\Sigma_t} L_k,
$$
along inverse curvature flows, together with some generalized Sobolev inequalities for $L_k$ (see \cite[Theorem 1A]{Guan-Wang2004}), they proved optimal Sobolev inequalities for {\em horospherical convex} hypersurfaces in $\H^n$ (i.e., hypersurfaces with all principal curvatures $\k_i\geq 1$, which will also be called {\em h-convex}) hypersurfaces). Recently, the second and the third authors \cite{Hu-Li2018} generalized their results to the hypersurfaces with nonnegative sectional curvature (i.e., $\k_i\k_j\geq 1$ for all distinct $i,j$) in hyperbolic space.
\begin{thmB}[\cite{Ge-Wang-Wu2014},\cite{Hu-Li2018}]
Let $0<2k< n-1$. If $\Sigma$ is a hypersurface with nonnegative sectional curvature in $\H^n$, then
\begin{align}\label{1.2}
\int_\Sigma L_k \geq C_{n-1}^{2k}(2k)!\omega_{n-1}^\frac{2k}{n-1}|\Sigma|^{\frac{n-1-2k}{n-1}}.
\end{align}
Equality holds in (\ref{1.2}) if and only if $\Sigma$ is a geodesic sphere.
\end{thmB}
It should be noticed that if $2k=n-1$, (\ref{1.2}) is an equality by the Gauss-Bonnet-Chern theorem \cite{Chern1944,Chern1945}. The nonnegativity of sectional curvature of the hypersurface plays an essential role in establishing the monotonicity of $Q(t)$ and the lower bound of $\lim_{t\ra\infty} Q(t)$. Based on Theorem B, we have the following Alexandrov-Fenchel inequalities for curvature integrals $\int_{\Sigma}p_{2k}$ and quermassintegrals $W_{2k+1}(\Omega)$ in terms of the area of the hypersurface, where $0<2k\leq n-1$.
\begin{thmC}[\cite{Ge-Wang-Wu2014},\cite{Hu-Li2018}]
Let $0<2k\leq n-1$.
\begin{enumerate}[(i)]
\item If $\Sigma$ is a hypersurface with nonnegative sectional curvature in $\H^n$, then
\begin{align}\label{1.3}
\frac{\int_{\Sigma}p_{2k}}{\omega_{n-1}} \geq \frac{|\Sigma|}{\omega_{n-1}}\left[1+\(\frac{|\Sigma|}{\omega_{n-1}}\)^{-\frac{2}{n-1}}\right]^k.
\end{align}
Equality holds in (\ref{1.3}) if and only if $\Sigma$ is a geodesic sphere in $\H^n$.
\item If $\Omega$ is a smooth bounded domain in $\H^n$ with boundary $\Sigma$ having nonnegative sectional curvature, then
\begin{align}\label{1.4}
\frac{W_{2k+1}(\Omega)}{\omega_{n-1}}\geq \frac{1}{n}\sum_{i=0}^{k}\frac{n-1-2k}{n-1-2i}C_k^i\(\frac{|\Sigma|}{\omega_{n-1}}\)^\frac{n-1-2i}{n-1}.
\end{align}
Equality holds in (\ref{1.4}) if and only if $\Sigma$ is a geodesic sphere in $\H^n$.
\end{enumerate}
\end{thmC}

By using the convergence result of the inverse mean curvature flow in $\S^n$ due to Makowski and Scheuer \cite{Makowski-Scheuer2013}, Wei and Xiong \cite{Wei-Xiong2015} proved that strictly convex hypersurfaces in $\S^n$ (i.e., those with all principal curvatures $\k_i>0$) also satisfy optimal Sobolev inequalities:
\begin{thmD}[\cite{Wei-Xiong2015}]
Let $0<2k< n-1$. If $\Sigma$ is a strictly convex hypersurface in $\S^n$, then
\begin{align}\label{1.5}
\int_\Sigma L_k \geq C_{n-1}^{2k}(2k)!\omega_{n-1}^\frac{2k}{n-1}|\Sigma|^{\frac{n-1-2k}{n-1}},
\end{align}
Equality holds in (\ref{1.5}) if and only if $\Sigma$ is a geodesic sphere.
\end{thmD}

Let $(\Sigma,g)$ be a strictly convex hypersurface in $\S^n$. Then the dual hypersurface $(\~\Sigma,\~g)$ to $\Sigma$ via the Gauss map is also a strictly convex hypersurface in $\S^n$ (see Section 2.3 for details). The $k$-th Gauss-Bonnet curvature $L_k$ on $(\~\Sigma,\~g)$ is
\begin{align*}
L_k(\~\k)=C_{n-1}^{2k}(2k)!\sum_{i=0}^{k}C_k^i p_{2k-2i}(\k^{-1})=C_{n-1}^{2k}(2k)!\sum_{i=0}^{k}C_k^i\frac{p_{n-1-2i}(\k)}{p_{n-1}(\k)},
\end{align*}
and $d\mu_{\~g}=p_{n-1}(\k)d\mu_g$. Applying Theorem D to $(\~\Sigma,\~g)$, we obtain
\begin{equation}\label{1.6}
\begin{split}
\sum_{i=0}^{k}C_k^i\int_{\Sigma} p_{n-1-2i}(\k) d\mu_g=&\frac{1}{C_{n-1}^{2k}(2k)!}\int_{\~\Sigma} L_k(\~\k) d\mu_{\~g} \\
\geq &\omega_{n-1}^\frac{2k}{n-1}|\~\Sigma|_{\~g}^{\frac{n-1-2k}{n-1}} \\
=&\omega_{n-1}^\frac{2k}{n-1}\(\int_{\Sigma}p_{n-1}(\k)d\mu_g\)^{\frac{n-1-2k}{n-1}}.
\end{split}
\end{equation}

It is natural to ask whether or not an inequality similar to (\ref{1.6}) also holds for hypersurfaces in $\H^n$. There also exists a one-to-one correspondence from closed, strictly convex hypersurfaces in $\H^n$ to closed, strictly convex, spacelike hypersurfaces in $\S_1^n$, where $\S_1^n$ is the $n$-dimensional de Sitter space of index $1$, see \cite[Chapter 10]{Gerhardt2006}. However, Theorem B can not apply to the dual hypersurfaces in de Sitter space. Motivated by this observation, in this article we will first prove the following Alexandrov-Fenchel inequalities for curvature integrals $\int_{\Sigma}p_{n-1-2k}$ and quermassintegrals $W_{n-1-(2k+1)}(\Omega)$ in terms of the total curvature $\int_{\Sigma}p_{n-1}$ of the hypersurface, where $0<2k\leq n-1$.
\begin{theo}\label{main-theo-1}
\begin{enumerate}[(i)]
\item Let $0<2k\leq n-1$. If $\Sigma$ is a strictly convex hypersurface in $\H^n$, then
\begin{align}\label{1.7}
\frac{\int_{\Sigma}p_{n-1-2k}}{\omega_{n-1}} \leq \frac{\int_{\Sigma} p_{n-1}}{\omega_{n-1}}\left[ 1- \(\frac{\int_\Sigma p_{n-1}}{\omega_{n-1}}\)^{-\frac{2}{n-1}}\right]^k.
\end{align}
Equality holds in (\ref{1.7}) if and only if $\Sigma$ is a geodesic sphere in $\H^n$.
\item Let $1<2k+1\leq n-1$. If $\Omega$ is a smooth bounded domain in $\H^n$ with strictly convex boundary $\Sigma$, then
\begin{align}\label{1.8}
\frac{W_{n-1-(2k+1)}(\Omega)}{\omega_{n-1}}\leq \frac{k+1}{C_n^2} \int_{1}^{\frac{\int_{\Sigma}p_{n-1}}{\omega_{n-1}}} \(1-s^{-\frac{2}{n-1}}\)^k ds.
\end{align}
Equality holds in (\ref{1.8}) if and only if $\Sigma$ is a geodesic sphere in $\H^n$.
\end{enumerate}
\end{theo}

We should mention that (\ref{1.7}) and (\ref{1.8}) are equivalent, see Remark \ref{rem-4.3}. As a corollary, we solve the following Alexandrov-Fenchel inequalities for the quermassintegrals, which solves an isoperimetric type problem posed by Gao-Hug-Schneider \cite{Gao-Hug-Schneider2003}.
\begin{cor}Let $1<2k+1\leq n-1$. If $\Omega$ is a smooth bounded domain in $\H^n$ with strictly convex boundary, then
    \begin{align}\label{1.9}
    W_{n-1-(2k+1)}(\Omega) \geq f_{n-1-(2k+1)} \circ f_{n-2}^{-1}(W_{n-2}(\Omega)).
    \end{align}
Equality holds in (\ref{1.9}) if and only if $\Omega$ is a geodesic ball. Here $f_{j}:[0,\infty)\ra \R_{+}$ is a monotone function defined by
$$
f_{j}(r)=W_{j}(B_r), \quad j=0,1,\cdots,n,
$$
the $j$-th quermassintegral for the geodesic ball of radius $r$, and $f_{j}^{-1}$ is the inverse function of $f_{j}$. In other words, the minimum of $W_{n-1-(2k+1)}$ among the domains with strictly convex boundary in $\H^n$ and fixed $W_{n-2}$ is achieved by geodesic balls.
\end{cor}

Inspired by the inequality (\ref{1.6}) for strictly convex hypersurfaces in $\S^n$, we can also prove a new family of Alexandrov-Fenchel type inequalities for h-convex hypersurfaces in $\H^n$, which can be compared with Theorem B.
\begin{theo}\label{main-theo-2}
Let $0<2k<n-1$. If $\Sigma$ is a h-convex hypersurface in $\H^n$, then
\begin{align}\label{1.10}
\sum_{i=0}^{k}C_k^i (-1)^i \int_{\Sigma}p_{n-1-2i} \geq \omega_{n-1}^{\frac{2k}{n-1}} \(\int_\Sigma p_{n-1} \)^{\frac{n-1-2k}{n-1}}.
\end{align}
Equality holds in (\ref{1.10}) if and only if $\Sigma$ is a geodesic sphere in $\H^n$.
\end{theo}

When $k=1$, (\ref{1.10}) coincides with (\ref{1.7}), i.e.,
\begin{align*}
\int_{\Sigma}p_{n-3}\leq \int_{\Sigma}p_{n-1}\left[ 1-\(\frac{\int_\Sigma p_{n-1}}{\omega_{n-1}}\)^{-\frac{2}{n-1}}\right].
\end{align*}
When $k=2$, the inequalities (\ref{1.7}) and (\ref{1.10}) are
\begin{align*}
\int_{\Sigma}p_{n-5}\leq \int_{\Sigma}p_{n-1}\left[ 1-\(\frac{\int_\Sigma p_{n-1}}{\omega_{n-1}}\)^{-\frac{2}{n-1}}\right]^2,
\end{align*}
and
\begin{align*}
\int_{\Sigma} \(p_{n-1}-2p_{n-3}+p_{n-5}\) \geq \omega_{n-1}^{\frac{4}{n-1}} \(\int_\Sigma p_{n-1} \)^{\frac{n-5}{n-1}}.
\end{align*}
From these, we get
\begin{align*}
\int_{\Sigma}p_{n-3}\leq &\int_{\Sigma} \frac{1}{2}\(p_{n-1}+p_{n-5}\)-\frac{1}{2}\omega_{n-1}^{\frac{4}{n-1}} \(\int_\Sigma p_{n-1} \)^{\frac{n-5}{n-1}} \\
\leq &\frac{1}{2}\int_{\Sigma} p_{n-1} \left\{ 1+\left[ 1-\(\frac{\int_\Sigma p_{n-1}}{\omega_{n-1}}\)^{-\frac{2}{n-1}}\right]^2-\(\frac{\int_\Sigma p_{n-1}}{\omega_{n-1}}\)^{-\frac{4}{n-1}} \right\} \\
=&\int_{\Sigma} p_{n-1} \left[ 1-\(\frac{\int_\Sigma p_{n-1}}{\omega_{n-1}}\)^{-\frac{2}{n-1}}\right].
\end{align*}
Hence (\ref{1.10}) is a refinement of (\ref{1.7}) under the stronger condition of h-convexity.

With the help of (\ref{1.7}), we will use the inverse mean curvature flow to prove the following Alexandrov-Fenchel inequality for strictly convex hypersurfaces in hyperbolic space, which was proved by Wang and Xia \cite{Wang-Xia2014} under the stronger condition that $\Sigma$ is h-convex.
\begin{theo}\label{main-theo-3}
If $\Sigma$ is a strictly convex hypersurface in $\H^n$, then
\begin{align}\label{1.11}
\int_{\Sigma}p_{n-1} \geq |\Sigma|\left[1+\(\frac{|\Sigma|}{\omega_{n-1}}\)^{-\frac{2}{n-1}} \right]^\frac{n-1}{2}.
\end{align}
Equality holds in (\ref{1.11}) if and only if $\Sigma$ is a geodesic sphere in $\H^n$.
\end{theo}
\begin{rem}
Theorem \ref{main-theo-3} gives an affirmative answer to the case ($2k+1=n-1$) in the conjecture proposed in \cite[Conjecture 16]{Hu-Li2018}.
\end{rem}

By the one-to-one correspondence between strictly convex hypersurface $(\Sigma,g)$ in $\H^n$ and strictly convex, spacelike hypersurface $(\~\Sigma,\~g)$ in $\S^n_1$, Theorem \ref{main-theo-1}(i) and Theorem \ref{main-theo-2} can be reformulated as the geometric inequalities for hypersurfaces in de Sitter space $\S^n_1$.
\begin{cor}
\begin{enumerate}[(i)]
\item Let $0<2k\leq n-1$. If $(\~\Sigma,\~g)$ is a strictly convex, spacelike hypersurface in $\S^n_1$, then
\begin{align}\label{1.12}
\frac{\int_{\~\Sigma}p_{2k}(\~\k)d\mu_{\~g}}{\omega_{n-1}} \leq \frac{|\~\Sigma|_{\~g}}{\omega_{n-1}}\left[ 1- \(\frac{|\~\Sigma|_{\~g}}{\omega_{n-1}}\)^{-\frac{2}{n-1}}\right]^k,
\end{align}
where $\~\k$ is the principal curvature of $(\~\Sigma,\~g)$ in $\S^n_1$ and $d\mu_{\~g}$ is the volume form, respectively. Equality holds in (\ref{1.12}) if and only if $\~\Sigma$ is a geodesic sphere in $\S^n_1$.
\item Let $0<2k<n-1$. If $(\~\Sigma,\~g)$ is a spacelike hypersurface with all principal curvatures $\~\k\in \{\l\in \R^{n-1}~|~0<\l_i\leq 1\}$ in $\S^n_1$, then
\begin{align}\label{1.13}
\sum_{i=0}^{k}C_k^i (-1)^i \int_{\~\Sigma}p_{2i}(\~\k)d\mu_{\~g} \geq \omega_{n-1}^{\frac{2k}{n-1}} |\~\Sigma|_{\~g}^{\frac{n-1-2k}{n-1}}.
\end{align}
Equality holds in (\ref{1.13}) if and only if $\~\Sigma$ is a geodesic sphere in $\S^n_1$.
\end{enumerate}
\end{cor}

The paper is organized as follows: In Section 2, we collect basic concepts and facts about integral geometry, the total curvature and the Gauss-Bonnet curvature $L_k$, Gauss map and dual hypersurfaces in sphere and hyperbolic space.

In Section 3, we recall the smooth convergence result (Proposition \ref{pro-convergence-HMCF}) of the harmonic mean curvature flow (HMCF) for strictly convex hypersurfaces in hyperbolic space, which has been investigated by Xu \cite{Xu2010}, see also Yu \cite{Yu2016}. We show that the inner radius and outer radius of the evolving hypersurface $\Sigma_t$ is comparable as it shrinks to a point. The main difficulty here is we only assume the initial hypersurface is strictly convex, so we do not have the remarkable property that its inner radius and outer radius are comparable for h-convex hypersurface. To overcome this obstacle, we project the domain $\Omega_t$ to the unit ball $B_1(0)$ in Euclidean space $\R^n$ via the Beltrami-Klein ball model of the hyperbolic space. We use the pinching estimate for hypersurfaces $\Sigma_t$ in $\H^n$ and the contracting property of the HMCF to show that the pinching estimate also holds for the hypersurfaces $\~\Sigma_t$ in $B_1(0)$. An argument of the first author \cite{Andrews1994-1} shows that the inner radius and outer radius of $\~\Sigma_t$ is comparable as $t\ra T^\ast$, where $T^\ast$ is the maximal existence time of the HMCF. The bounded distortion then implies that the inner radius and outer radius of $\Sigma_t$ is comparable as $t\ra T^\ast$. This idea was previously used for the curvature contraction flows by Gerhardt \cite{Gerhardt2015-2} in the sphere, and later by Yu \cite{Yu2016} in hyperbolic space. With this remarkable property, we obtain the limits of quermassintegrals and curvature integrals as $t\ra T^\ast$. We also use the tensor maximum principle (proved by the first author \cite{Andrews2007}) to show that the h-convexity is preserved along the HMCF.

In Section 4, we give the proof of Theorem \ref{main-theo-1}. To show (\ref{1.7}), we first consider the functional:
\begin{equation}\label{1.14}
\begin{split}
P_k(t):=&\(\frac{\int_{\Sigma_t}p_{n-1}}{\omega_{n-1}}\)^{-\frac{n-1-2k}{n-1}}\frac{\int_{\Sigma_t}p_{n-1-2k}}{\omega_{n-1}}\\
&-\(\frac{\int_{\Sigma_t}p_{n-1}}{\omega_{n-1}}\)^\frac{2k}{n-1}\left[1-\(\frac{\int_{\Sigma_t} p_{n-1}}{\omega_{n-1}}\)^{-\frac{2}{n-1}} \right]^k.
\end{split}
\end{equation}
One of the crucial points is to establish the monotonicity of the functional $P_k(t)$ along the HMCF. To achieve this, we need to use an induction argument, which is inspired by the proof of Alexandrov-Fenchel inequalities for strictly convex hypersurfaces in Euclidean space, see Section 7. By the limits of curvature integrals as $t\ra T^\ast$, we have $\lim_{t\ra T^\ast}P_k(t)=0$, which completes the proof of (\ref{1.7}). With the help of (\ref{1.7}) and the limits of quermassintegrals as $t\ra T^\ast$, we get (\ref{1.8}).

In Section 5, we give the proof of Theorem \ref{main-theo-2}. We consider the functional
\begin{align}\label{1.15}
Q_k(t):=\(\int_{\Sigma_t}p_{n-1}\)^{-\frac{n-1-2k}{n-1}}\left[\sum_{i=0}^{k}C_k^i(-1)^i\int_{\Sigma_t} p_{n-1-2i}\right].
\end{align}
By a crucial observation due to Ge-Wang-Wu \cite{Ge-Wang-Wu2014}, we show that $Q_k(t)$ is monotone decreasing along the HMCF if the initial hypersurface is h-convex. By the limits of curvature integrals as $t\ra T^\ast$, we finish the proof of (\ref{1.10}) by showing that $\lim_{t\ra T^\ast}Q_k(t)=0$. If the equality holds in (\ref{1.10}), we have $Q_k(t)\equiv 0$ for all $t\in [0,T^\ast)$. For $t>0$, the flow hypersurface $\Sigma_t$ of the HMCF is strictly h-convex, we show that $\Sigma_t$ is totally umbilical and hence it is a geodesic sphere. Finally, the initial hypersurface is smoothly approximated by a family of geodesic spheres, and it must be a geodesic sphere in $\H^n$. The similar idea has been used by the second and third authors in \cite{Hu-Li2018}.

In Section 6, we use the new Alexandrov-Fenchel inequality in Theorem \ref{main-theo-1} and inverse mean curvature flow (IMCF) to prove Theorem \ref{main-theo-3}. We consider the functional
\begin{align}\label{1.16}
Q(t):=\(\frac{|\Sigma_t|}{\omega_{n-1}}\)^{-1}\left[\frac{|\Sigma_t|}{\omega_{n-1}}-\(\frac{\int_{\Sigma_t}p_{n-1}}{\omega_{n-1}}\)\left[1-\(\frac{\int_{\Sigma_t}p_{n-1}}{\omega_{n-1}}\)^{-\frac{2}{n-1}}\right]^\frac{n-1}{2}\right].
\end{align}
With the help of the inequality (\ref{1.7}), we show $Q(t)$ is monotone increasing along the IMCF. By the convergence result of Gerhardt \cite{Gerhardt2011}, we finish the proof by showing that $\lim_{t\ra \infty}Q(t)=0$.

In Section 7, we prove the Alexandrov-Fenchel inequalities for strictly convex hypersurfaces in Euclidean space, which has its own interests. In Section 8, we use the mean curvature flow to prove a new Heintze-Karcher type inequality for hypersurfaces with positive Ricci curvature in hyperbolic space.

Our choice of the HMCF is a curvature contraction flow, which is quite different from the usual choice of inverse curvature flows (the expanding flows) \cite{Brendle-Hung-Wang2015,Ge-Wang-Wu2014,Hu-Li2018,Li-Wei2017,Li-Wei-Xiong2014,Lima-Girao2016} or the quermassintegral preserving flows \cite{Andrews-Chen-Wei2018,Andrews-Wei2017,Wang-Xia2014}, etc. We highlight that the curvature contraction flows for hypersurfaces in hyperbolic space may only require the convexity conditions weaker than h-convexity on the initial hypersurface. Moreover, the analysis of limiting hypersurfaces is simple, and does not require the application of optimal Sobolev inequalities as in inverse curvature flows \cite{Brendle-Hung-Wang2015,Ge-Wang-Wu2014,Hu-Li2018,Li-Wei-Xiong2014,Lima-Girao2016}, etc.

\section*{Acknowledgements}
The first author was supported by Australian Laureate Fellowship FL150100126 of the Australian Research Council. The second author was supported by China Postdoctoral Science Foundation No.2018M641317. The third author was supported by NSFC grant No.11671224,11831005.

\section{Preliminaries and Notations}
\subsection{Curvature integrals and Quermassintegrals}
We recall some basic concepts and formulas in integral geometry. We refer to Santal\'o's book \cite{Santalo1976}, see also Schneider \cite{Schneider1993} or Solanes \cite{Solanes2003,Solanes2006}. The space form $(\mathbb{M}^{n}(\e),\-g)$ is an $n$-dimensional simply connected Riemannian manifold of constant curvature $\e\in \{-1,0,1\}$. Let $\Omega$ be a compact domain with smooth boundary $\Sigma=\partial\Omega$ in $\mathbb{M}^{n}(\e)$. Then $(\Sigma,g)$ is a closed hypersurface in $\mathbb{M}^{n}(\e)$, where $g$ is the induced metric on $\Sigma$. Let $\-\nabla$ be the connection on $(\mathbb{M}^{n}(\e),\-g)$ and $\nu$ the unit outward normal on $\Sigma$, respectively. The second fundamental form $h$ of $\Sigma$ is defined by
\begin{align*}
h(X,Y)=\langle \-\nabla_X \nu, Y\rangle,
\end{align*}
for any tangent vector fields $X,Y$ on $\Sigma$. For an orthonormal basis $\{e_1,\cdots,e_{n-1}\}$ of $\Sigma$, the second fundamental form is $h=(h_{ij})$ and the Weingarten tensor is $\mathcal{W}=(h_i^j)=(g^{jk}h_{ki})$, respectively. The principal curvatures $\k=(\k_1,\cdots,\k_{n-1})$ are the eigenvalues of $\mathcal{W}$.

Let $\s_k$ be the $k$-th elementary symmetric function $\s_k:\R^{n-1}\ra \R$ defined by
\begin{align*}
\s_k(\l)=\sum_{i_1<\cdots<i_k}\l_{i_1}\cdots \l_{i_k}, \quad \text{for}~ \l=(\l_1,\cdots,\l_{n-1})\in \R^{n-1}.
\end{align*}
We also take $\s_0=1$ by convention. The Garding cone is defined as
\begin{align*}
\G_k^+=\{\l\in \R^{n-1}~|~\s_j(\l)>0,\forall j\leq k\}.
\end{align*}
Let $p_k(\l)=\frac{\s_k(\l)}{C_{n-1}^k}$ be the (normalized) $k$-th elementary symmetric function. We have the well-known Newton-MacLaurin inequalities (see e.g. \cite[Lemma 5]{Guan2014}).
\begin{lem}\label{lem-2.1}
Let $1\leq k\leq n-1$. For $\l \in \G_k^+$, we have
\begin{align}\label{2.1}
p_{k+1}p_{k-1}\leq p_k^2, \quad p_{k+1}^{\frac{1}{k+1}} \leq p_k^\frac{1}{k}.
\end{align}
Moreover, the above equalities hold if and only if $\l=\a(1,\cdots,1)$ for some $\a>0$.
\end{lem}

The {\em normalized $k$-th order mean curvature} of $\Sigma$ is defined by
\begin{align}\label{2.2}
p_k(x):=p_k(\k(x)), \quad x\in \Sigma, \quad  0\leq k \leq n-1.
\end{align}
and the {\em curvature integrals} are defined by
\begin{align*}
V_{n-1-j}(\Omega):=\int_{\Sigma}p_j, \quad 0\leq j\leq n-1.
\end{align*}
In particular, $V_0(\Omega)=\int_{\Sigma}p_{n-1}$ is called the {\em total curvature}.

For a convex domain $\Omega\subset \mathbb{M}^{n}(\e)$, the {\em quermassintegrals} are defined by
\begin{align}\label{2.3}
W_r(\Omega):=\frac{(n-r)\omega_{r-1}\cdots\omega_{0}}{n\omega_{n-2}\cdots \omega_{n-r-1}}\int_{\mathcal{L}_r}\chi(L_r\cap \Omega)dL_r, \quad r=1,\cdots,n-1,
\end{align}
where $\mathcal{L}_r$ is the space of $r$-dimensional totally geodesic subspaces $L_r$ in $\mathbb{M}^{n}(\e)$, and $dL_r$ is the natural measure on $\mathcal{L}_r$ which is invariant under the isometry group of $\mathbb{M}^{n}(\e)$. The function $\chi$ is defined to be $1$ if $L_r\cap \Omega\neq \emptyset$ and to be $0$ otherwise. Furthermore, we set $W_0(\Omega)=\mrm{Vol}(\Omega)$ and $W_n(\Omega)=\frac{\omega_{n-1}}{n}$. It is clear from (\ref{2.3}) and the definition of volume of the domain, the quermassintegrals $W_r$, $r=0,\cdots,n-1$, are increasing under set inclusion, i.e.,
\begin{align}\label{2.4}
W_r(\Omega_1) \leq W_r(\Omega_2), \quad \text{if}\quad \Omega_1 \subset \Omega_2.
\end{align}
The curvature integrals and quermassintegrals in $\mathbb{M}^{n}(\e)$ are related by the following recursive formulas (see \cite[Proposition 7]{Solanes2006}):
\begin{align}\label{2.5}
V_{n-1-j}(\Omega)=n\(W_{j+1}(\Omega)-\e\frac{j}{n-j+1}W_{j-1}(\Omega)\),\quad j=1,\cdots,n-1.
\end{align}
From the recursive formulas (\ref{2.5}), we can express $W_r$ as a linear combination of curvature integrals (see e.g. \cite{Santalo1976} or \cite[Corollary 8]{Solanes2006}, see also \cite{Wang-Xia2014}):
\begin{enumerate}[(i)]
\item For $1\leq r\leq n$ and $r$ being odd,
\begin{align}\label{2.6}
W_r(\Omega)=\frac{1}{n}\sum_{i=0}^{\frac{r-1}{2}}\frac{(r-1)!!(n-r)!!}{(r-1-2i)!!(n-r+2i)!!}\e^{i}\int_{\Sigma}p_{r-1-2i};
\end{align}
\item For $1\leq r \leq n$ and $r$ being even,
\begin{equation}\label{2.7}
\begin{split}
W_r(\Omega)=&\frac{1}{n}\sum_{i=0}^{\frac{r}{2}-1}\frac{(r-1)!!(n-r)!!}{(r-1-2i)!!(n-r+2i)!!}\e^{i}\int_{\Sigma}p_{r-1-2i}\\
&+\frac{(r-1)!!(n-r)!!}{n!!}\e^{\frac{r}{2}}\mrm{Vol}(\Omega).
\end{split}
\end{equation}
\end{enumerate}
Here the notation $k!!$ means the product of all odd (even) integers up to odd (even) $k$. In the particular case $r=n$, it is just the Gauss-Bonnet-Chern theorem \cite{Chern1944,Chern1945} for strictly convex hypersurfaces in $\mathbb{M}^n(\e)$, see \cite{Solanes2006}.

\subsection{Gauss-Bonnet curvatures}
Given an $(n-1)$-dimensional Riemannian manifold $(\Sigma,g)$, the Gauss-Bonnet curvature $L_k$, where $k\leq \frac{n-1}{2}$, is defined by
\begin{align}\label{2.8}
L_k=\frac{1}{2^k} \d_{j_1 j_2 \cdots j_{2k-1} j_{2k}}^{i_1 i_2 \cdots i_{2k-1} i_{2k}} R_{i_1 i_2}{}^{j_1 j_2} \cdots R_{i_{2k-1}i_{2k}}{}^{j_{2k-1}j_{2k}},
\end{align}
where $R_{ij}{}^{kl}$ is the Riemannian curvature tensor in the local coordinates with respect to $g$, and the generalized Kronecker delta is defined by
\begin{align*}
\d_{i_1i_2\cdots i_r}^{j_1j_2\cdots j_r}=\det\(\begin{matrix}
     \d_{i_1}^{j_1} & \d_{i_1}^{j_2} & \cdots & \d_{i_1}^{j_r} \\
     \d_{i_2}^{j_1} & \d_{i_2}^{j_2} & \cdots & \d_{i_2}^{j_r} \\
      \vdots        &    \vdots      & \vdots & \vdots \\
     \d_{i_r}^{j_1} & \d_{i_r}^{j_2} & \cdots & \d_{i_r}^{j_r}
     \end{matrix}\)
\end{align*}
For a hypersurface $\Sigma$ in $\mathbb{M}^{n}(\e)$, the Gauss equation is
\begin{align*}
R_{ij}{}^{kl}=(h_i^k h_j^l-h_i^l h_j^k)+\e(\d_i^k\d_j^l -\d_i^l \d_j^k).
\end{align*}
A direct calculation then gives the relation between $L_k$ and $p_k$:
\begin{equation}\label{2.9}
\begin{split}
L_k=&\d_{j_1 j_2 \cdots j_{2k-1} j_{2k}}^{i_1 i_2 \cdots i_{2k-1} i_{2k}}(h_{i_1}^{j_1}h_{i_2}^{j_2}+\e\d_{i_1}^{j_1}\d_{i_2}^{j_2})\cdots (h_{i_ {2k-1}}^{j_{2k-1}}h_{i_{2k}}^{j_{2k}}+\e\d_{i_{2k-1}}^{j_{2k-1}}\d_{i_{2k}}^{j_{2k}}) \\
=&\sum_{i=0}^{k}C_k^i \e^i(n-2k)(n-2k+1)\cdots (n-1-2k+2i)(2k-2i)!C_{n-1}^{2k-2i}p_{2k-2i} \\
=&\sum_{i=0}^{k}C_k^i \e^i\frac{(n-1)!}{(n-1-2k)!}p_{2k-2i}\\
=&C_{n-1}^{2k}(2k)!\sum_{i=0}^{k}C_k^i \e^i p_{2k-2i}.
\end{split}
\end{equation}
For simplicity, we denote
\begin{align}\label{2.10}
\~L_k=\sum_{i=0}^{k}C_k^i \e^i p_{2k-2i}, \quad \~N_k=\sum_{i=0}^{k}C_k^i \e^i p_{2k-2i+1}.
\end{align}

\subsection{Gauss maps and Dual hypersurfaces}
The Hadamard theorem (see do Carmo and Warner \cite{doCarmo-Warner1970}, see also Chapters 9 and 10 in Gerhardt's book \cite{Gerhardt2006}) states that for a closed, strictly convex, connected orientable immersed hypersurface $\Sigma$ in $\mathbb{M}^{n}(\e)$, it is necessarily embedded, $\Sigma$ is diffeomorphic to $\S^{n-1}$ and $\Sigma$ bounds a convex body $\Omega$ in $\mathbb{M}^{n}(\e)$.

In \cite[Chapter 9]{Gerhardt2006}, Gerhardt considered a one-to-one correspondence between closed, strictly convex hypersurfaces in the sphere via the Gauss map. More precisely, if $\Sigma$ is a closed, connected, strictly convex hypersurface given by an immersion
$$
x: M^{n-1} \ra \Sigma \subset \S^n,
$$
then the Hadamard theorem states that $\Sigma$ is embedded, diffeomorphic to $\S^{n-1}$, contained in an open hemisphere and is the boundary of a convex body $\Omega$ in $\S^n$. If we consider $\Sigma$ as a codimension $2$ submanifold in $\R^{n+1}$ such that
\begin{align*}
x_{ij}=-g_{ij}x-h_{ij}\~x,
\end{align*}
where $\~x\in T_x(\R^{n+1})$ is the unit outward normal vector $v\in T_x(\S^n)$ and the map
$$
\~x:M^{n-1} \ra \~\Sigma \subset \S^n
$$
is called the {\em Gauss map} of $\Sigma$. By \cite[Theorem 9.2.5]{Gerhardt2006}, the Gauss map $\~x$ is the embedding of a closed, strictly convex hypersurface $\~\Sigma$ in $\S^n$. We call $\~\Sigma$ the {\em dual hypersurface} of $\Sigma$. Viewing $\~\Sigma$ as a codimension $2$ submanifold in $\R^{n+1}$, its Gauss formula is
\begin{align*}
\~x_{ij}=-\~g_{ij}\~x-\~h_{ij}x,
\end{align*}
where $x$ is the embedding of $\Sigma$ which also represents the outward normal vector of $\~\Sigma$. The induced metrics $g$ (resp. $\~g$), the second fundamental forms $h$ (resp. $\~h$) and the principal curvatures $\k_i$ (resp. $\~\k_i$) of $\Sigma$ (resp. $\~\Sigma$) are closely related:
\begin{align*}
\~g_{ij}=h_i^k h^{m}_{k}g_{mj}, \quad \~h_{ij}=h_{ij}, \quad \~\k_i=\k_i^{-1}.
\end{align*}

In the similar spirit, Gerhardt \cite[Chapter 10]{Gerhardt2006} also established a one-to-one correspondence from the closed, strictly convex hypersurfaces in hyperbolic space $\H^n$ to the closed, strictly convex, spacelike hypersurfaces in de Sitter space $\S^n_1$. More precisely, if $\R^{n,1}$ is the $(n+1)$-dimensional Minkowski spacetime, that is the vector space $\R^{n+1}$ endowed with the Minkowski spacetime metric $\langle \cdot,\cdot,\rangle$ by
$$
\langle x,x\rangle =-(x^0)^2+\sum_{i=1}^{n} (x^i)^2,
$$
for any vector $x=(x^0,x^1,\cdots,x^n)\in \R^{n+1}$, where $x^0$ is the time coordinate. The hyperbolic space is
\begin{align*}
\H^n=\{x\in \R^{n,1} ~:~\langle x,x\rangle =-1,x^0>0\},
\end{align*}
and de Sitter space is
\begin{align*}
\S^n_1=\{x\in \R^{n,1} ~:~\langle x,x\rangle=1\}.
\end{align*}
If $\Sigma$ is a closed, connected, strictly convex hypersurface in $\H^n$ given by an immersion
$$
x:M^{n-1} \ra \Sigma \subset \H^n,
$$
then the Hadamard theorem states that $\Sigma$ is embedded, diffeomorphic to $\S^{n-1}$, contained in an open hemisphere and is the boundary of a convex body $\Omega$ in $\H^n$. If we consider $\Sigma$ as a codimension $2$ submanifold in $\R^{n,1}$ such that
\begin{align*}
x_{ij}=g_{ij}x-h_{ij}\~x,
\end{align*}
where $\~x\in T_x(\R^{n,1})$ is the unit outward normal vector $v\in T_x(\H^n)$ and the Gauss map
$$
\~x:M^{n-1} \ra \~\Sigma \subset \S^n_1,
$$
is an embedding of a closed, strictly convex, spacelike hypersurface $\~\Sigma$ in $\S^n_1$, which is the dual hypersurface of $\Sigma$ (see \cite[Theorem 10.4.4]{Gerhardt2006}). Viewing $\~\Sigma$ as a codimension $2$ submanifold in $\R^{n,1}$, its Gauss formula is
\begin{align*}
\~x_{ij}=-\~g_{ij}\~x+\~h_{ij}x,
\end{align*}
where $x$ is the embedding of $\Sigma$ which also represents the unit outward normal vector of $\~\Sigma\subset \S^n_1$, and the second fundamental form $\~h_{ij}$ is defined with respect to the future directed normal vector, where the time orientation of $\S^n_1$ is inherited from $\R^{n,1}$. The induced metrics $g$ (resp. $\~g$), the second fundamental forms $h$ (resp. $\~h$) and the principal curvatures $\k_i$ (resp. $\~\k_i$) of $\Sigma$ (resp. $\~\Sigma$) are also closely related:
\begin{align}\label{2.15}
\~g_{ij}=h_i^k h^{m}_{k}g_{mj}, \quad \~h_{ij}=h_{ij}, \quad \~\k_i=\k_i^{-1}.
\end{align}

\section{Harmonic mean curvature flow in hyperbolic space}
Let $X_0:M^{n-1}\ra \H^n$ be a smooth embedding such that $\Sigma= X_0(M)$ is a closed smooth hypersurface in $\H^n$. We consider a smooth family of immersions $X:M^{n-1}\times [0,T) \ra \H^n$ satisfying
\begin{equation}\label{3.1}
\begin{split}
\left\{\begin{aligned}
\frac{\partial}{\partial t}X(x,t)=&-\frac{p_{n-1}}{p_{n-2}}(\mathcal{W}(x,t))\nu(x,t),\\
X(\cdot,0)=&X_0(\cdot),
\end{aligned}\right.
\end{split}
\end{equation}
where $\Sigma_t=X(M,t)$ is a family of hypersurfaces in $\H^n$, $\nu$ is the unit outward normal to $\Sigma_t$. This flow (\ref{3.1}) is called the {\em harmonic mean curvature flow} (briefly, HMCF).
In contrast to the inverse curvature flows \cite{Brendle-Hung-Wang2015,Ge-Wang-Wu2014,Hu-Li2018,Li-Wei2017,Li-Wei-Xiong2014,Lima-Girao2016} or the quermassintegral preserving flows \cite{Andrews-Chen-Wei2018,Andrews-Wei2017,Wang-Xia2014}, etc., the harmonic mean curvature flow is a curvature contraction flow. Curvature contraction flows have been widely used in proving various geometric inequalities. In \cite{Andrews1996}, the first author investigated the affine curvature flow and proved the affine isoperimetric inequality in Euclidean space. Topping \cite{Topping1998} used the curve shortening flow to prove an isoperimetric-type inequality on surfaces. Schulze \cite{Schulze2008} applied the $H^k$-flow to prove the isoperimetric inequality for domains with smooth boundary in $\R^{n+1}$, where $n\leq 7$. He also gave a new proof for the Euclidean isoperimetric inequality on complete, simply connected $3$-dimensional manifolds with nonpositive sectional curvature, which was previously proved by Kleiner \cite{Kleiner1992}.

Our method also provides a new approach to proving geometric inequalities in hyperbolic space. Under the assumption that the initial hypersurface is strictly convex and satisfies the condition $\k_i H>n-1$ for each $i$, Huisken \cite{Huisken1986} proved that the mean curvature flow converges in finite time to a round point. The first author \cite{Andrews1994} proved the smooth convergence results for the flow of h-convex hypersurfaces in hyperbolic space, with speed given by functions with argument $\k_i-1$, in particular the {\em (shifted) harmonic mean curvature flow}. Later, Xu (see \cite[Theorem 3]{Xu2010}) proved the smooth convergence of the HMCF (\ref{3.1}) for strictly convex hypersurfaces in complete, simply connected $n$-dimensional manifolds with nonpositive sectional curvature. Chen and the first author \cite{Andrews-Chen2015} proved the smooth convergence of mean curvature flow for hypersurfaces with positive Ricci curvature in hyperbolic space. Recently, Yu \cite{Yu2016} proved the smooth convergence for a general class of curvature contraction flows in hyperbolic space.

The volume preserving mean curvature flow in hyperbolic space was first studied by Cabezas-Rivas and Miquel \cite{Cabezas-Rivas-Miquel2007}. Later, Makowski \cite{Makowski2012} proved smooth convergence of curvature contraction flow in hyperbolic space with a global term chosen to preserve the curvature integrals of the evolving hypersurfaces, provided the initial hypersurface is strictly h-convex. Later, Wang-Xia \cite{Wang-Xia2014} proved the smooth convergence of a similar flow of h-convex hypersurfaces in hyperbolic space, which preserves the quermassintegrals of the evolving domains. Recently, Wei and the first author \cite{Andrews-Wei2017} proved the smooth convergence of more general quermassintegral preserving curvature flows in $\H^n$. More recently, Chen, Wei and the first author \cite{Andrews-Chen-Wei2018} proved the smooth convergence of a volume preserving flow for hypersurfaces with positive sectional curvature in $\H^n$.

\subsection{Smooth convergence of HMCF}
The smooth convergence of the HMCF (\ref{3.1}) for strictly convex hypersurfaces in hyperbolic space has been proved by Xu \cite[Theorem 3]{Xu2010}, see also Yu \cite[Theorem 1.2]{Yu2016}.
\begin{pro}\label{pro-convergence-HMCF}
If $\Sigma$ is a strictly convex hypersurface in $\H^n$, then there exists a unique smooth solution of the HMCF (\ref{3.1}) on a maximal time interval $[0,T^\ast)$, and the hypersurfaces $\Sigma_t$ converge uniformly to a round point $p_0\in \H^n$ as $t\ra T^\ast$, in the sense that the rescaled flow converges smoothly to a round sphere. Moreover, the flow hypersurface $\Sigma_t$ is strictly convex for $t\in [0,T^\ast)$.
\end{pro}

A major ingredient in the proof of the smooth convergence of the HMCF (\ref{3.1}) is the pinching estimate on the principal curvatures of the evolving hypersurfaces $\Sigma_t$, which was obtained by Xu \cite[Theorem 21]{Xu2010}).
\begin{lem}\label{lem-3.2}
If the initial hypersurface $\Sigma$ is strictly convex, then along the HMCF (\ref{3.1}) the evolving hypersurface $\Sigma_t$ satisfies the pinching estimate
$$
\k_{n-1}(x,t) \leq C \k_1(x,t), \quad \forall (x,t)\in M \times [0,T^\ast),
$$
where $\k_1 \leq \cdots \leq \k_{n-1}$ and $C$ depends only on the initial hypersurface $\Sigma$.
\end{lem}

\subsection{Comparability of inner radius and outer radius}
In order to estimate the limiting behavior of various functionals as $t\ra T^\ast$, the crucial part is to show that the inner radius and outer radius are comparable as $t\ra T^\ast$. Recall that the inner radius $\rho_{-}$ and outer radius $\rho_{+}$ of a bounded domain $\Omega_t$ with boundary $\Sigma_t$ in $\H^n$ is defined by
\begin{align*}
\rho_{-}(t):=&\sup \left\{ \rho~:~ \text{$B_\rho(p)$ is enclosed by $\Omega_t$ for some $p\in \H^n$}\right\},\\
\rho_{+}(t):=&\inf \left\{ \rho~:~ \text{$B_\rho(p)$ encloses $\Omega_t$ for some $p\in \H^n$}\right\},
\end{align*}
where $B_\rho(p)$ denotes the geodesic ball of radius $\rho$ about $p$ in $\H^n$. The comparability of inner radius and outer radius of the evolving domain $\Omega_t$ is satisfied if the evolving hypersurface $\Sigma_t$ is h-convex, since there exists a uniform constant $c>0$ such that $\rho_{+}\leq c(\rho_{-}+\rho_{-}^{1/2})$, see e.g. \cite[Theorem 1]{Borisenko-Miquel1999} or \cite[Theorem 2.1]{Andrews-Wei2017}. This remarkable property has been used widely in curvature contraction flows of h-convex hypersurfaces in hyperbolic space with a global term, see \cite{Cabezas-Rivas-Miquel2007,Makowski2012,Wang-Xia2014,Andrews-Wei2017}. Recently, Chen, Wei and the first author \cite{Andrews-Chen-Wei2018} considered volume preserving flows of hypersurfaces with positive sectional curvature in hyperbolic space. To obtain this property without h-convexity, they used an Alexandrov reflection argument to bound the diameter of the domain $\Omega_t$ enclosed by the evolving hypersurface $\Sigma_t$. Then they projected the domain $\Omega_t$ to the unit ball in Euclidean space $\R^n$ via the Beltrami-Klein ball model of the hyperbolic space.
The upper bound on the diameter of $\Omega_t$ implies that this map has bounded distortion. This together with the preservation of the volume of $\Omega_t$ gives a uniform lower bound on the inner radius of $\Omega_t$.

Here we also project the domain $\Omega_t$ to the unit ball $B_1(0)$ in Euclidean space $\R^n$ via the Beltrami-Klein ball model of the hyperbolic space. We use the contracting property of the HMCF and show that the inner radius and outer radius are comparable as $t\ra T^\ast$. As the hypersurface $\Sigma_t$ shrinks to a point as $t\ra T^\ast$, the inner radius and outer radius both approach zero as $t\ra T^\ast$. By the monotonicity of $W_r(\Omega_t)$ under set inclusion, we obtain the limits $\lim_{t\ra T^\ast}W_r(\Omega_t)$ for $0\leq r\leq n$. By the recursive formulas (\ref{2.5}), we also achieve the limits $\lim_{t\ra T^\ast}\int_{\Sigma_t}p_j$ for $0\leq j\leq n-1$. These limits will be used in calculating the limits of $P_k(t)$ and $Q_k(t)$.

For a fixed point $x_0\in \H^n$, the hyperbolic metric in the geodesic polar coordinates about $x_0$ can be expressed as
\begin{align}\label{3.2}
d\-s^2=d\rho^2+\sinh^2 \rho\,\s_{ij} d\t^i d\t^j,
\end{align}
where $\s_{ij}$ is the standard metric of $\S^{n-1}$. The geodesic spheres about $x_0$ are totally umbilical. The induced metric, second fundamental form, and the principal curvatures of the geodesic sphere $S_{\rho}=\{x^0=\rho\}$ are given by
\begin{align}\label{3.3}
\-g_{ij}=\sinh^2 \rho \s_{ij},\quad h_{ij}=\coth \rho \-g_{ij}, \quad \-\k_i=\coth \rho.
\end{align}
When the initial hypersurface is a geodesic sphere $S_{\rho_0}$, the evolving hypersurfaces are all geodesic spheres with the same center and their radii $\T=\T(t)$ satisfying the equation
\begin{align*}
\frac{d\T}{dt}=\coth \T, \quad \T(0)=\rho_0.
\end{align*}
More precisely, it can be rewritten as
\begin{align*}
\T(t,T^\ast)=\mrm{arccosh} (e^{T^\ast-t}),
\end{align*}
where $T^\ast=\ln \cosh \rho_0$.

The following lemma follows from the same argument as in \cite[Lemma 6.1]{Gerhardt2015-2}, see also \cite[Lemma 6.2]{Yu2016}.
\begin{lem}\label{lem-3.3}
Let $\Sigma_t$ be a solution of (\ref{3.1}) on a maximal time interval $[0,T^\ast)$, and represent $\Sigma_t$ for a fixed $t\in [0,T^\ast)$ as a graph in polar coordinates about $x_0 \in \Omega_t$:
\begin{align*}
\Sigma_t = \mrm{graph}\left(u(t,\cdot)\right)=\{\rho=u(t,\t)~|~\t \in \S^{n-1} \}.
\end{align*}
Then
\begin{align}\label{3.4}
\inf_{\Sigma_t} u \leq \T(t,T^\ast) \leq \sup_{\Sigma_t} u.
\end{align}
\end{lem}

Now we project the domain $\Omega_t$ in hyperbolic space $\H^n$ to the unit ball in Euclidean space $\R^n$ as in \cite{Andrews-Wei2017}. An embedding $X:M^{n-1}\ra \H^n$ induces an embedding $Y:M^{n-1}\ra B_1(0)\subset \R^n$ by
\begin{align}\label{3.5}
X=\frac{(1,Y)}{\sqrt{1-|Y|^2}}.
\end{align}
Let $g_{ij}^X$, $h_{ij}^X$ and $g_{ij}^X$, $h_{ij}^Y$ be the induced metrics and second fundamental forms of $X(M)\subset \H^n$ and $Y(M)\subset \R^n$, respectively, and $N\in \R^n$ be the unit normal vector of $Y(M)$. We have
\begin{align}\label{3.6}
h_{ij}^X=\frac{h_{ij}^Y}{\sqrt{(1-|Y|^2)(1-\langle N,Y\rangle^2)}},
\end{align}
and
\begin{align}\label{3.7}
g_{ij}^X=\frac{1}{1-|Y|^2}\left[ g_{ij}^Y+\frac{\langle Y,\partial_i Y\rangle \langle Y,\partial_j Y\rangle}{1-|Y|^2} \right].
\end{align}
Since each $\Sigma_t=X(M,t)$ is strictly convex in $\H^n$, the equation (\ref{3.6}) implies that $Y_t=Y(M,t)$ is strictly convex in $\R^n$ as well.

\begin{lem}\label{lem-3.4}
Let $\{Y_t\}_{t\in [0,T^\ast)}$ be the corresponding image of $\Sigma_t$ in $B_1(0)\subset \R^n$. Then there exists a positive constant $C<1$ depending only on the initial hypersurface $\Sigma$, such that
\begin{align}\label{3.8}
|Y(\cdot,t)|\leq C<1, \quad t\in [0,T^\ast).
\end{align}
\end{lem}
\begin{proof}
We first prove that the diameter of $\Omega_t$ is bounded as the proof in \cite[Proposition 1]{Andrews-Chen2015}. Let $d$ be the hyperbolic distance from any fixed point in $\H^n$. Then $d$ is smooth where it is nonzero, and the evolution of $d$ is
\begin{align*}
\frac{\partial d}{\partial t}=D d \(-\frac{p_{n-1}}{p_{n-2}}\nu\).
\end{align*}
At a point where the spatial maximum of $d$ is attained, we have
\begin{align*}
0=\nabla_i d=D d(\partial_i),
\end{align*}
so that $\nu$ points in the radial direction, and
\begin{align*}
0\geq \nabla_i \nabla_j d =D^2 d(\partial_i,\partial_j)-Dd(h_{ij}\nu).
\end{align*}
Since $\nu$ is radial, we have $D^2 d(\partial_i,\partial_j)=(\cosh d) g_{ij}$, which implies
$$
\k_i \geq \coth d, \quad \text{for all $i$}.
$$
As $\frac{p_{n-1}}{p_{n-2}}$ is an increasing function of the principal curvatures, we get
$$
\frac{p_{n-1}}{p_{n-2}}(\k) \geq \frac{p_{n-1}}{p_{n-2}}(\coth d,\cdots,\coth d)=\coth d>0.
$$
The maximum principle \cite[Lemma 3.5]{Hamilton1986} implies that the maximum of $d$ is non-increasing, and hence $d\leq d_0:=\sup_{x\in M} d(X(x,0))$. Thus, the diameter of $\Omega_t$ is bounded. Let $\~\Omega_t\subset B_1(0)$ be the corresponding image of $\Omega_t$. Then $\~\Omega_t$ is a convex domain in $B_1(0)$. Then the diameter bound of $\Omega_t$ implies the diameter bound of $\~\Omega_t$. In particular, there exists a positive constant $C<1$ such that
\begin{align*}
|Y(\cdot,t)|\leq C<1, \quad t\in [0,T^\ast).
\end{align*}
\end{proof}

By Lemma \ref{lem-3.4} and (\ref{3.7}), the induced metric $g_{ij}^X$ and $g_{ij}^Y$ are comparable. More precisely, there exists a constant $\d\in (0,1)$ depending only on the initial hypersurface $\Sigma$ such that
\begin{align}\label{3.9}
\d^2 g_{ij}^X \leq g_{ij}^Y \leq g_{ij}^X.
\end{align}
By (\ref{3.6}), together with Lemma \ref{lem-3.2}, we have pinching estimate on the principal curvatures of $\{Y_t\}_{t\in [0,T^\ast)}$ in $B_1(0)\subset \R^n$:
\begin{align}\label{3.10}
\~\k_{n-1}(x,t) \leq C \~\k_1(x,t), \quad \forall (x,t)\in M \times [0,T^\ast),
\end{align}
where $\~\k_1 \leq \cdots \leq \~\k_{n-1}$ are the principal curvatures of $Y_t$ and $C$ depends only on the initial hypersurface $\Sigma$.

Now we choose $x_0=(1,0)\in \Omega_t$ to be the center of the inner ball of $\Omega_t\subset \H^n$ and let $x_0$ be the center of the geodesic polar coordinates of $\H^n$. Then the center of the Euclidean inner ball of $\~\Omega_t\subset B_1(0)$ is $0$. Let $\~\rho_{-}(t)$ and $\~\rho_{+}(t)$ be the inner radius and outer radius of $\~\Omega_t$ in $B_1(0)$, respectively. In \cite{Andrews1994}, the first author proved the comparability of the inner radius and outer radius for strictly convex hypersurfaces in Euclidean space, provided that the pinching estimate (\ref{3.10}) holds. With the help of this property, we show the comparability of the inner radius and outer radius of $\Sigma_t$ as $t\ra T^\ast$.
\begin{lem}\label{lem-3.5}
There exist positive constants $C$ and $\eta$, depending only on the initial hypersurface $\Sigma$, such that
\begin{align}\label{3.11}
\rho_{+}(t)  \leq C\rho_{-}(t), \quad \forall t\in [T^\ast-\eta,T^\ast).
\end{align}
\end{lem}
\begin{proof}
The pinching estimate (\ref{3.10}) in Euclidean space and \cite[Theorem 5.1 and Lemma 5.4]{Andrews1994-1} imply that there exists a uniform constant $c>0$ such that
\begin{align*}
\~\rho_{+}(t) \leq c\~\rho_{-}(t).
\end{align*}
Hence $\~\Omega_t$ is contained in the Euclidean ball $B_{2c\~\rho_{-}(t)}(0)$. The Euclidean metric over $B_1(0)$ in the geodesic polar coordinates can be expressed as
\begin{align}\label{3.12}
d\~s^2=dr^2+r^2 \s_{ij} d\t^i d\t^j,
\end{align}
where $r$ is the Euclidean distance in $B_1(0)$. As $Y_t$ is strictly convex in $B_1(0)$, it can be expressed as a graph $\~u$ over $\S^{n-1}$:
\begin{align*}
Y_t= \mrm{graph} \~u = \{r=\~u(t,\t)~|~\t \in \S^{n-1} \}.
\end{align*}
By (\ref{3.5}), the geodesic sphere $S_\rho=\{x^0=\rho\}$ in hyperbolic space $\H^n$ corresponds to the geodesic sphere $\~S_r=\{|Y|=r\}$ in $B_1(0)$, and hence $r=\tanh \rho$. By taking
\begin{align*}
\~\Theta(t,T^\ast):=\tanh \Theta(t,T^\ast)=\sqrt{1-e^{-2(T^\ast-t)}},
\end{align*}
it follows from (\ref{3.4}) that
\begin{align*}
\inf_{Y_t}\~u \leq \~\Theta(t,T^\ast) \leq \sup_{Y_t}\~u.
\end{align*}
Thus, we have
\begin{align*}
2c\~\rho_{-}(t)\leq 2c\~\T(t,T^\ast).
\end{align*}
Observe that $\~\T(t,T^\ast)\ra 0$ as $t\ra T^\ast$, we can pick a sufficiently small constant $\eta>0$ such that
\begin{align*}
2c\~\T(t,T^\ast)\leq 1, \quad  \forall t\in [T^\ast-\eta,T^\ast).
\end{align*}
Let $\rho(t)=\mrm{arctanh}(2c\~\rho_{-}(t))$. Then $\Omega_t \subset B_{\rho(t)}(x_0) \subset \H^n$. Since $0\leq 2c\~\T(t,T^\ast)\leq 1$, we get
\begin{align*}
2c\~\rho_{-}(t) \leq \rho(t) \leq 4c \~\rho_{-}(t), \quad \~\rho_{-}(t)\leq \rho_{-}(t),
\end{align*}
where the latter inequality follows from (\ref{3.9}). Finally, we obtain
$$
\rho_{+}(t)\leq \rho(t) \leq 4c\~\rho_{-}(t) \leq 4c \rho_{-}(t),
$$
which completes the proof.
\end{proof}

By the comparability of the inner radius and outer radius of $\Sigma_t$ as $t\ra T^\ast$, we can estimate the quermassintegrals and curvature integrals of the evolving hypersurfaces of the HMCF (\ref{3.1}) as $t\ra T^\ast$.
\begin{lem}\label{lem-3.6}
Let $\Sigma$ be a strictly convex hypersurface in $\H^n$. Let $\Sigma_t$, $t\in [0,T^\ast)$ be the solution of the HMCF (\ref{3.1}) with the initial hypersurface $\Sigma$, and $\Omega_t$ the domain enclosed by $\Sigma_t$. Then we have
\begin{equation}\label{3.13}
\begin{split}
\lim_{t\ra T^\ast}W_k(\Omega_t)=\left\{\begin{aligned}
&0, \quad    &0\leq k\leq n-1;\\
&\frac{\omega_{n-1}}{n}, \quad &k=n,
\end{aligned}\right.
\end{split}
\end{equation}
and
\begin{equation}\label{3.14}
\begin{split}
\lim_{t\ra T^\ast}\int_{\Sigma_t}p_j=\left\{\begin{aligned}
&0, \quad    &0\leq j\leq n-2;\\
&\omega_{n-1}, \quad &j=n-1,
\end{aligned}\right.
\end{split}
\end{equation}
\end{lem}
\begin{proof}
Let $\Sigma_t$ be the evolving hypersurface for $t\in [0,T^\ast)$ and $\Omega_t$ the domain enclosed by $\Sigma_t$, respectively. By Proposition \ref{pro-convergence-HMCF}, we have $\rho_{-}(t)\ra 0$ as $t\ra T^\ast$. By (\ref{3.11}) in Lemma \ref{lem-3.5}, we know that $\rho_{+}(t)\ra 0$ as $t\ra T^\ast$.
By the monotonicity of $W_r$ under set inclusion, we have
\begin{align}\label{3.15}
W_r(B_{\rho_{-}(t)})\leq W_r(\Omega_t)\leq W_r(B_{\rho_{+}(t)}).
\end{align}
Since $p_j(B_\rho)=\coth^j \rho$, it follows from (\ref{2.6}) and (\ref{2.7}) that if $1\leq r\leq n-1$ and $r$ is odd,
$$
W_r(B_\rho)=\frac{1}{n}\sum_{i=0}^{\frac{r-1}{2}}\frac{(r-1)!!(n-r)!!}{(r-1-2i)!!(n-r+2i)!!}\omega_{n-1}\coth^{r-1-2i}\rho\sinh^{n-1}\rho,
$$
while if $1\leq r\leq n-1$ and $r$ is even,
\begin{align*}
W_r(B_\rho)=&\frac{1}{n}\sum_{i=0}^{\frac{r}{2}-1}\frac{(r-1)!!(n-r)!!}{(r-1-2i)!!(n-r+2i)!!}\omega_{n-1}\coth^{r-1-2i}\rho\sinh^{n-1}\rho\\
&+(-1)^\frac{r}{2}\frac{(r-1)!!(n-r)!!}{n!!}\int_0^\rho \omega_{n-1}\sinh^{n-1}(s)ds.
\end{align*}
If $r=0$, then $W_0(B_\rho)=\Vol(B_\rho)=\int_{0}^{\rho}\omega_{n-1}\sinh^{n-1}(s)ds$. If $r=n$, it follows from the Gauss-Bonnet-Chern theorem that $W_n(B_\rho)=\frac{\omega_{n-1}}{n}$. By the explicit expression for $W_r(B_\rho)$, we have
\begin{align*}
\lim_{\rho \ra 0}W_r (B_{\rho})=\left\{\begin{aligned}
&0, \quad   &0\leq r\leq n-1;\\
&\frac{\omega_{n-1}}{n}, \quad &r=n.
\end{aligned}\right.
\end{align*}
By (\ref{3.15}), we obtain (\ref{3.13}). Then (\ref{3.14}) follows from (\ref{2.5}) and (\ref{3.13}).
\end{proof}

\subsection{Preserving of h-convexity}
We will use the tensor maximum principle to prove that {\em h-convexity} is preserved along the HMCF. The tensor maximum principle was first proved by Hamilton \cite{Hamilton1982} and was generalized by the first author \cite{Andrews2007}.
\begin{theo}[\cite{Andrews2007}]\label{theo-Andrews}
Let $S_{ij}$ be a smooth time-varying symmetric tensor field on a closed manifold $M$, satisfying
\begin{align}\label{Andrews-evolution-1}
\frac{\partial}{\partial t}S_{ij} =a^{kl} \nabla_k \nabla_l S_{ij}+u^k\nabla_k S_{ij}+N_{ij},
\end{align}
where $a^{kl}$ and $u^k$ are smooth, $\nabla$ is a (possibly time-dependent) smooth symmetric connection, and $a^{kl}$ is positive definite everywhere. Suppose that
\begin{align}\label{Andrews-evolution-2}
N_{ij}v^iv^j+\sup_{\L} 2 a^{kl} (2\L_k^p \nabla_l S_{ip} v^i-\L_k^p \L_l^q S_{pq}) \geq 0,
\end{align}
where $S_{ij}\geq 0$ and $S_{ij}v^j=0$. If $S_{ij}$ is positive definite everywhere on $M$ at $t=0$, then it is positive definite on $M\times[0,T]$.
\end{theo}

Denote $\Psi(\mathcal{W})=\frac{p_{n-1}}{p_{n-2}}(\k(\mathcal{W}))$, and
$\dot{\Psi}^{kl}$, $\ddot{\Psi}^{kl,pq}$ denote the derivatives of $\Psi$ with respect to the components of $\mathcal{W}=(h_i^j)$.
\begin{lem}\label{lem-3.8}
Along the HMCF (\ref{3.1}), the Weingarten tensor $(h_i^j)$ of $\Sigma_t$ evolves by
\begin{equation}\label{Weingarten-evolution}
\frac{\partial}{\partial t} h_i^j=\dot{\Psi}^{kl}\nabla_k \nabla_l h_i^j +\ddot{\Psi}^{kl,pq} \nabla_i h_{kl}\nabla^j h_{pq}+(\dot{\Psi}^{kl}(h^2)_{kl}+\dot{\Psi}^{kl}g_{kl})h_i^j-2\Psi\d_i^j.
\end{equation}
\end{lem}
\begin{proof}
By \cite[Lemma 2.4]{Andrews-Wei2017}, the Weigarten tensor $h_{i}^j$ of $\Sigma_t$ evolves by
\begin{align*}
\frac{\partial}{\partial t} h_i^j=&\dot{\Psi}^{kl}\nabla_k \nabla_l h_i^j +\ddot{\Psi}^{kl,pq} \nabla_i h_{kl}\nabla^j h_{pq}+(\dot{\Psi}^{kl}(h^2)_{kl}+\dot{\Psi}^{kl}g_{kl})h_i^j\\
&-(\dot{\Psi}^{kl} h_{kl}-\Psi)(h^2)_i^j-(\dot{\Psi}^{kl} h_{kl}+\Psi)\d_i^j.
\end{align*}
Then (\ref{Weingarten-evolution}) follows from $\dot{\Psi}^{kl} h_{kl}=\Psi$, since $\Psi$ is homogenous of degree $1$.
\end{proof}

Without resorting to the constant rank theorem as before (see \cite{Yu2016}), here we follow the spirit of the recent work of Wei and the first author \cite{Andrews-Wei2017} to prove that h-convexity is preserved along the HMCF (\ref{3.1}).
\begin{lem}\label{lem-3.9}
Let $\Sigma_t$, $t\in [0,T)$ be a solution of the HMCF (\ref{3.1}) in $\H^n$. If the initial hypersurface $\Sigma$ is h-convex, then the evolving hypersurface $\Sigma_t$ is strictly h-convex for $t\in (0,T)$.
\end{lem}
\begin{proof}
We show that h-convexity is preserved along the HMCF, and that strict h-convexity holds for $t>0$. Define $S_{ij}:=h_{i}^j-\d_i^j$. Then $h$-convexity is equivalent to $S_{ij}\geq 0$. By (\ref{Weingarten-evolution}), the tensor $S_{ij}$ evolves by
\begin{equation}\label{Sij-evolution}
\begin{split}
\frac{\partial}{\partial t}S_{ij}=&\dot{\Psi}^{kl}\nabla_k \nabla_l S_{ij}+\ddot{\Psi}^{kl,pq}\nabla_i h_{kl} \nabla_j h_{pq} +(\dot{\Psi}^{kl}(h^2)_{kl}+\dot{\Psi}^{kl}g_{kl})S_{ij} \\
&+\dot{\Psi}^{kl}((h^2)_{kl}+g_{kl}-2h_{kl})\d_{i}^{j}.
\end{split}
\end{equation}
To apply the tensor maximum principle, we need to show that (\ref{Andrews-evolution-2}) holds provided that $S_{ij}\geq 0$ and $S_{ij}v^{j}=0$. Let $(x_0,t_0)$ be the point where $S_{ij}$ has a null vector $v$. By continuity, we can assume that $h_i^j$ has all eigenvalues distinct and in increasing order at $(x_0,t_0)$, that is $\k_{n-1}>\k_{n-2}>\cdots>\k_1$. The null eigenvector condition $S_{ij}v^j=0$ implies that $v=e_1$ and $S_{11}=\k_1-1=0$ at $(x_0,t_0)$. The terms in (\ref{Sij-evolution}) which contains $S_{ij}$ satisfies the null vector condition.
Denote $\psi(\k)=\frac{p_{n-1}(\k)}{p_{n-2}(\k)}$, then a direct calculation gives
\begin{align*}
\dot{\psi}^k=&\frac{1}{n-1}\(\frac{p_{n-1}(\k)}{p_{n-2}(\k)}\)^2 \frac{1}{\k_k^2}=\frac{1}{n-1}\frac{\psi^2}{\k_k^2},\\
\ddot{\psi}^{kl}=&\frac{2}{(n-1)^2}\(\frac{p_{n-1}(\k)}{p_{n-2}(\k)}\)^3\frac{1}{\k_k^2\k_l^2}-\frac{2}{n-1}\(\frac{p_{n-1}(\k)}{p_{n-2}(\k)}\)^2\frac{1}{\k_k^3}\d_{kl}\\
=&2\psi^{-1}\dot{\psi}^k \dot{\psi}^l-2\frac{\dot{\psi}^k}{\k_k}\d_{kl}.
\end{align*}
For the last term in (\ref{Sij-evolution}) we have
\begin{align*}
\dot{\Psi}^{kl}((h^2)_{kl}+g_{kl}-2h_{kl})=&\sum_{k}\dot{\psi}^{k}(\k_k^2+1-2\k_k)\\
                                          =&\frac{\psi^2}{n-1}\sum_{k}\frac{(\k_k-1)^2}{\k_k^2} \geq 0.
\end{align*}
Thus, it remains to show that
\begin{align*}
Q_1:=\ddot{\Psi}^{kl,pq}\nabla_1 h_{kl}\nabla_1 h_{pq}+2\sup_{\L} \dot{\Psi}^{kl}(2\L_k^p\nabla_l S_{1p}-\L_k^p\L_l^q S_{pq}) \geq 0.
\end{align*}
Note that $S_{11}=0$ and $\nabla_k S_{11}=0$ at $(x_0,t_0)$, the supremum over $\L$ can be explicitly computed as follows.
\begin{align*}
 &2\dot{\Psi}^{kl}(2 \L_k^p \nabla_l S_{1p}-\L_k^p \L_l^q S_{pq})\\
=&2\sum_{k=1}^{n-1}\sum_{p=2}^{n-1}\dot{\psi}^k (2\L_k^p \nabla_k S_{1p}-(\L_k^p)^2 S_{pp}) \\
=&2\sum_{k=1}^{n-1}\sum_{p=2}^{n-1}\dot{\psi}^k\(\frac{(\nabla_k S_{1p})^2}{S_{pp}}-\(\L_k^p-\frac{\nabla_k S_{1p}}{S_{pp}}\)^2 S_{pp}\).
\end{align*}
Thus the supremum is obtained by taking $\L_k^p=\frac{\nabla_k S_{1p}}{S_{pp}}$. The required inequality for $Q_1$ becomes
\begin{align*}
Q_1=\ddot{\Psi}^{kl,pq}\nabla_1 h_{kl}\nabla_1 h_{pq}+2\sum_{k=1}^{n-1}\sum_{p=2}^{n-1}\dot{\psi}^k\frac{(\nabla_k S_{1p})^2}{S_{pp}}\geq 0.
\end{align*}
By the Codazzi equation we have $\nabla_1 S_{1p}=\nabla_1 h_{1p}=\nabla_p h_{11}=0$ at $(x_0,t_0)$, we have
\begin{align*}
Q_1=&\ddot{\psi}^{kl}\nabla_1 h_{kk}\nabla_1 h_{ll}+2\sum_{k>l} \frac{\dot{\psi}^k-\dot{\psi}^l}{\k_k-\k_l}(\nabla_1 h_{kl})^2 +2\sum_{k>1,l>1}\frac{\dot{\psi}^k}{\k_l-1}(\nabla_1 h_{kl})^2\\
=&2\psi^{-1}(\nabla_1 \Psi)^2-2\sum_{k>1}\frac{\dot{\psi}^k}{\k_k}(\nabla_1 h_{kk})^2+2\sum_{k>l}\frac{\dot{\psi}^k-\dot{\psi}^l}{\k_k-\k_l}(\nabla_1 h_{kl})^2\\
&+2\sum_{k>1,l>1}\frac{\dot{\psi}^k}{\k_l-1}(\nabla_1 h_{kl})^2\\
\geq &-2\sum_{k>1}\frac{\dot{\psi}^k}{\k_k}(\nabla_1 h_{kk})^2-2\sum_{k\neq l>1}\frac{\dot{\psi}^k}{\k_l}(\nabla_1 h_{kl})^2+2\sum_{k>1,l>1}\frac{\dot{\psi}^k}{\k_l-1}(\nabla_1 h_{kl})^2 \\
= &2\sum_{k>1,l>1}\(\frac{\dot{\psi}^k}{\k_l-1}-\frac{\dot{\psi}^k}{\k_l}\)(\nabla_1 h_{kl})^2 \geq 0.
\end{align*}
Here we have used the fact $\nabla_k h_{11}=0$ and the following identity
\begin{align*}
2\sum_{k>l}\frac{\dot{\psi}^k-\dot{\psi}^l}{\k_k-\k_l}(\nabla_1 h_{kl})^2=&-2\frac{\psi^2}{n-1}\sum_{k>l}\frac{\k_k+\k_l}{\k_k^2\k_l^2}(\nabla_1 h_{kl})^2\\
   =&-2\sum_{k>l}\(\frac{\dot{\psi}^k}{\k_l}+\frac{\dot{\psi}^l}{\k_k}\)(\nabla_1 h_{kl})^2\\
   =&-2\sum_{k\neq l>1}\frac{\dot{\psi}^k}{\k_l}(\nabla_1 h_{kl})^2.
\end{align*}
Thus, the tensor maximum principle (Theorem \ref{theo-Andrews}) implies that the h-convexity is preserved along the HMCF.

Finally, we show that $\Sigma_t$ is strictly h-convex for $t>0$. If this is not true, then there exists some interior point $(x_0,t_0)$ such that the smallest principal curvature is $1$. By the strong maximum principle, there exists a parallel vector field $v$ such that $S_{ij}v^iv^j=0$ on $\Sigma_{t_0}$. Then the smallest principal curvature is $1$ on $\Sigma_{t_0}$ everywhere. This contradicts with the fact that on any closed hypersurface in $\H^{n}$, there exists at least one point where all the principal curvatures are strictly larger than one. This completes the proof of Theorem \ref{pro-convergence-HMCF}.
\end{proof}

\section{Proof of Theorem \ref{main-theo-1}}
In this section, we give the proof of Theorem \ref{main-theo-1}. One can check using \cite{Reilly1973}  that along the HMCF (\ref{3.1}), we have
\begin{align}\label{4.1}
\frac{d}{dt}\int_{\Sigma} p_j=\int_{\Sigma} \left[ (n-j-1)p_{j+1}+jp_{j-1} \right]\(-\frac{p_{n-1}}{p_{n-2}}\),\quad 0\leq j\leq n-1.
\end{align}

We first prove the Alexandrov-Fenchel inequality (\ref{1.7}) for curvature integrals in Theorem \ref{main-theo-1}.
\begin{theo}\label{theo-4.1}
Let $0<2k\leq n-1$. If $\Sigma$ is a strictly convex hypersurface in $\H^n$, then
\begin{align}\label{4.2}
\frac{\int_{\Sigma}p_{n-1-2k}}{\omega_{n-1}} \leq \frac{\int_{\Sigma} p_{n-1}}{\omega_{n-1}}\left[ 1- \(\frac{\int_\Sigma p_{n-1}}{\omega_{n-1}}\)^{-\frac{2}{n-1}}\right]^k.
\end{align}
Equality holds in (\ref{4.2}) if and only if $\Sigma$ is a geodesic sphere in $\H^n$.
\end{theo}
\begin{proof}
We first prove (\ref{4.2}) for $k=1$, i.e.,
\begin{align}\label{4.3}
\frac{\int_{\Sigma}p_{n-3}}{\omega_{n-1}} \leq \frac{\int_{\Sigma} p_{n-1}}{\omega_{n-1}}\left[ 1- \(\frac{\int_\Sigma p_{n-1}}{\omega_{n-1}}\)^{-\frac{2}{n-1}}\right].
\end{align}
The variation formula (\ref{4.1}) and the Newton-MacLaurin inequality (\ref{2.1}) give
\begin{align}\label{4.4}
\frac{d}{dt}\int_{\Sigma} p_{n-1}=-(n-1)\int_{\Sigma} p_{n-1},
\end{align}
and
\begin{equation}\label{4.5}
\begin{split}
\frac{d}{dt}\int_{\Sigma} p_{n-3}=&-2\int_{\Sigma}p_{n-1}-(n-3)\int_{\Sigma}\frac{p_{n-1}p_{n-4}}{p_{n-2}}\\
\geq &-2\int_{\Sigma}p_{n-1}-(n-3)\int_{\Sigma}p_{n-3}.
\end{split}
\end{equation}
By Lemma \ref{lem-3.6}, we have
\begin{align}\label{4.6}
\lim_{t\ra T^\ast}\(\frac{\int_{\Sigma_t}p_{n-1}}{\omega_{n-1}}\)=1.
\end{align}
Combining with (\ref{4.4}), we get
\begin{align}\label{4.7}
\frac{\int_{\Sigma_t}p_{n-1}}{\omega_{n-1}}=e^{(n-1)(T^\ast-t)}\geq 1, \quad \forall t\in [0,T^\ast).
\end{align}
We define the functional
\begin{align*}
P_1(t):=&\(\frac{\int_{\Sigma}p_{n-1}}{\omega_{n-1}}\)^{-\frac{n-3}{n-1}}\left[\frac{\int_{\Sigma}p_{n-3}}{\omega_{n-1}}-\frac{\int_{\Sigma}p_{n-1}}{\omega_{n-1}}\(1-\(\frac{\int_\Sigma p_{n-1}}{\omega_{n-1}}\)^{-\frac{2}{n-1}} \)\right].
\end{align*}
It follows from (\ref{4.4}) and (\ref{4.5}) that
\begin{align*}
\frac{d}{dt}\(\frac{\int_{\Sigma} p_{n-1}}{\omega_{n-1}}\)= -(n-1)\(\frac{\int_{\Sigma} p_{n-1}}{\omega_{n-1}}\),
\end{align*}
and
\begin{align*}
&\frac{d}{dt}\left[\frac{\int_{\Sigma} p_{n-3}}{\omega_{n-1}}-\frac{\int_{\Sigma} p_{n-1}}{\omega_{n-1}}\(1-\(\frac{\int_{\Sigma} p_{n-1}}{\omega_{n-1}}\)^{-\frac{2}{n-1}}\)\right]\\
\geq &-(n-3)\left[\frac{\int_{\Sigma} p_{n-3}}{\omega_{n-1}}-\frac{\int_{\Sigma} p_{n-1}}{\omega_{n-1}}\(1-\(\frac{\int_{\Sigma} p_{n-1}}{\omega_{n-1}}\)^{-\frac{2}{n-1}}\)\right].
\end{align*}
Therefore, we have $\frac{d}{dt}P_1(t)\geq 0$. Together with Lemma \ref{lem-3.6}, we get $P_1(0)\leq \lim_{t\ra T^\ast}P_1(t)=0$, i.e.,
\begin{align}\label{4.9}
\(\frac{\int_{\Sigma}p_{n-1}}{\omega_{n-1}}\)^{-\frac{n-3}{n-1}}\left[\frac{\int_{\Sigma}p_{n-3}}{\omega_{n-1}}-\frac{\int_{\Sigma}p_{n-1}}{\omega_{n-1}}\(1-\(\frac{\int_\Sigma p_{n-1}}{\omega_{n-1}}\)^{-\frac{2}{n-1}} \)\right]\leq 0.
\end{align}
The inequality (\ref{4.3}) then follows from (\ref{4.7}) and (\ref{4.9}).

We prove (\ref{4.2}) for $k\geq 2$ by induction. We assume that (\ref{4.2}) holds for $k-1$, i.e.,
\begin{align}\label{4.10}
\(\frac{\int_{\Sigma}p_{n-1-2(k-1)}}{\omega_{n-1}}\)\leq  \(\frac{\int_\Sigma p_{n-1}}{\omega_{n-1}}\)\left[ 1- \(\frac{\int_\Sigma p_{n-1}}{\omega_{n-1}}\)^{-\frac{2}{n-1}}\right]^{k-1},
\end{align}
then we show that (\ref{4.2}) also holds for $k$. The variation formula (\ref{4.1}) and Newton-MacLaurin inequality (\ref{2.1}) give
\begin{equation}\label{4.11}
\begin{split}
\frac{d}{dt}\int_{\Sigma} p_{n-1-2k}=&-2k\int_{\Sigma}p_{n-2k}\frac{p_{n-1}}{p_{n-2}}-(n-1-2k)\int_{\Sigma} p_{n-2-2k}\frac{p_{n-1}}{p_{n-2}}\\
\geq &-2k\int_{\Sigma}p_{n+1-2k}-(n-1-2k)\int_{\Sigma} p_{n-1-2k}.
\end{split}
\end{equation}
We define the functional
\begin{align*}
P_k(t):=&\(\frac{\int_{\Sigma}p_{n-1}}{\omega_{n-1}}\)^{-\frac{n-1-2k}{n-1}}\frac{\int_{\Sigma}p_{n-1-2k}}{\omega_{n-1}}\\
&-\(\frac{\int_{\Sigma}p_{n-1}}{\omega_{n-1}}\)^\frac{2k}{n-1}\left[1-\(\frac{\int_\Sigma p_{n-1}}{\omega_{n-1}}\)^{-\frac{2}{n-1}} \right]^k.
\end{align*}
We claim that $\frac{d}{dt}P_k(t)\geq 0$. For simplicity, we take
$$
x(t)=\frac{\int_{\Sigma}p_{n-1}}{\omega_{n-1}},\quad y(t)=\frac{\int_{\Sigma}p_{n-1-2k}}{\omega_{n-1}}.
$$
It follows from (\ref{4.10}) and (\ref{4.11}) that
\begin{equation}\label{4.12}
\begin{split}
\frac{d}{dt}y=&\frac{d}{dt}\(\frac{\int_{\Sigma} p_{n-1-2k}}{\omega_{n-1}}\) \\
\geq &-2k\(\frac{\int_{\Sigma}p_{n-1-2(k-1)}}{\omega_{n-1}}\)-(n-1-2k)\(\frac{\int_{\Sigma} p_{n-1-2k}}{\omega_{n-1}}\)\\
\geq &-2k x\(1-x^{-\frac{2}{n-1}}\)^{k-1}-(n-1-2k)y.
\end{split}
\end{equation}
By (\ref{4.4}), we have $\frac{d}{dt}x=-(n-1)x$. A direct calculation gives
\begin{align*}
\frac{d}{dt}\left[ y-x\(1-x^{-\frac{2}{n-1}}\)^k \right]\geq -(n-1-2k)\left[y-x\(1-x^{-\frac{2}{n-1}}\)^k\right],
\end{align*}
and hence $\frac{d}{dt}P_k(t)\geq 0$. By Lemma \ref{lem-3.6} and the monotonicity of $P_k(t)$, we have
\begin{align*}
P_k(0) \leq \lim_{t\ra T^\ast} P_k(t)=0.
\end{align*}
This, together with (\ref{4.7}), shows that the inequality (\ref{4.2}) holds for $k$.

If the equality holds in (\ref{4.2}), then the equality in the Newton-MacLaurin inequality (\ref{2.1}) implies that $\Sigma$ is totally umbilical and hence it is a geodesic sphere in $\H^n$.
\end{proof}

Now we prove the Alexandrov-Fenchel inequality (\ref{1.7}) for quermassintegrals in Theorem \ref{main-theo-1}.
\begin{theo}\label{theo-4.2}
Let $1<2k+1\leq n-1$. If $\Omega$ be a smooth bounded domain in $\H^n$ with strictly convex boundary $\Sigma$, then
\begin{align}\label{4.13}
\frac{W_{n-1-(2k+1)}(\Omega)}{\omega_{n-1}}\leq \frac{k+1}{C_n^2} \int_{1}^{\frac{\int_{\Sigma}p_{n-1}}{\omega_{n-1}}} \(1-s^{-\frac{2}{n-1}}\)^k ds.
\end{align}
Equality holds in (\ref{4.13}) if and only if $\Sigma$ is a geodesic sphere in $\H^n$.
\end{theo}
\begin{proof}The variational formula for quermassintegrals (see \cite{Wang-Xia2014}) along the HMCF (\ref{3.1}) is
\begin{align*}
\frac{d}{dt}W_{n-1-(2k+1)}(\Omega)=-\frac{2k+2}{n}\int_{\Sigma}p_{n-1-(2k+1)}\frac{p_{n-1}}{p_{n-2}}.
\end{align*}
Together with the Newton-MacLaurin inequality (\ref{2.1}) and Theorem \ref{theo-4.1}, we get
\begin{equation}\label{4.14}
\begin{split}
\frac{d}{dt}\(\frac{W_{n-1-(2k+1)}(\Omega)}{\omega_{n-1}}\)\geq& -\frac{2k+2}{n}\(\frac{\int_{\Sigma}p_{n-1-2k}}{\omega_{n-1}}\)\\
\geq &-\frac{2k+2}{n}\(\frac{\int_\Sigma p_{n-1}}{\omega_{n-1}}\)\left[ 1- \(\frac{\int_\Sigma p_{n-1}}{\omega_{n-1}}\)^{-\frac{2}{n-1}}\right]^k.
\end{split}
\end{equation}
By (\ref{4.4}), we have
\begin{align*}
\frac{d}{dt}\int_1^{\frac{\int_\Sigma p_{n-1}}{\omega_{n-1}}} \(1-s^{-\frac{2}{n-1}}\)^k ds=-(n-1)\(\frac{\int_\Sigma p_{n-1}}{\omega_{n-1}}\)\left[1-\(\frac{\int_\Sigma p_{n-1}}{\omega_{n-1}}\)^{-\frac{2}{n-1}}\right]^k.
\end{align*}
Thus we have
\begin{align}\label{4.15}
&\frac{d}{dt}\left[\frac{W_{n-1-(2k+1)}(\Omega)}{\omega_{n-1}}-\frac{k+1}{C_n^2}\int_1^{\frac{\int_\Sigma p_{n-1}}{\omega_{n-1}}} \(1-s^{-\frac{2}{n-1}}\)^k \right]\geq 0.
\end{align}
By Lemma \ref{lem-3.6}, we have
\begin{align*}
\lim_{t\ra T^\ast}\left[\frac{W_{n-1-(2k+1)}(\Omega_t)}{\omega_{n-1}}-\frac{k+1}{C_n^2}\int_1^{\frac{\int_{\Sigma_t} p_{n-1}}{\omega_{n-1}}}\(1-s^{-\frac{2}{n-1}}\)^k \right]=0.
\end{align*}
This, together with (\ref{4.15}), gives (\ref{4.13}). If the equality holds in (\ref{4.13}), then the equality in the Newton-MacLaurin inequality (\ref{2.1}) implies that $\Sigma$ is totally umbilical and hence it is a geodesic sphere in $\H^n$.
\end{proof}

\begin{rem}\label{rem-4.3}
We should notice that Theorem \ref{theo-4.1} is equivalent to Theorem \ref{theo-4.2} in the following sense. From the proof of Theorem \ref{theo-4.2}, it is easy to see that (\ref{4.13}) can be deduced from (\ref{4.2}) in Theorem \ref{theo-4.1}. On the other hand, it follows from (\ref{2.5}) that
\begin{align}\label{4.16}
\frac{1}{n}\(\frac{\int_{\Sigma}p_{n-1-2k}}{\omega_{n-1}}\)=\(\frac{W_{n-2k}(\Omega)}{\omega_{n-1}}\)+\frac{n-1-2k}{2k+2}\(\frac{W_{n-2k-2}(\Omega)}{\omega_{n-1}}\).
\end{align}
By (\ref{4.13}) in Theorem \ref{theo-4.2}, we have
\begin{align*}
\frac{1}{n}\(\frac{\int_{\Sigma}p_{n-1-2k}}{\omega_{n-1}}\)=&\(\frac{W_{n-2k}(\Omega)}{\omega_{n-1}}\)+\frac{n-1-2k}{2k+2}\(\frac{W_{n-2k-2}(\Omega)}{\omega_{n-1}}\)\\
\leq &\frac{k}{C_n^2}\int_{1}^{\frac{\int_{\Sigma}p_{n-1}}{\omega_{n-1}}} \(1-s^{-\frac{2}{n-1}}\)^{k-1} ds\\
     &+\frac{(k+1)(n-1-2k)}{C_n^2 (2k+2)}\int_{1}^{\frac{\int_{\Sigma}p_{n-1}}{\omega_{n-1}}} \(1-s^{-\frac{2}{n-1}}\)^k ds\\
=&\frac{1}{C_n^2} \int_1^{\frac{\int_{\Sigma}p_{n-1}}{\omega_{n-1}}} \left[ k(1-s^{-\frac{2}{n-1}})^{k-1}+\frac{n-1-2k}{2}(1-s^{-\frac{2}{n-1}})^{k}\right] ds\\
=&\frac{1}{n}\(\frac{\int_{\Sigma}p_{n-1}}{\omega_{n-1}}\)\left[1-\(\frac{\int_{\Sigma}p_{n-1}}{\omega_{n-1}}\)^{-\frac{2}{n-1}}\right]^k,
\end{align*}
which gives (\ref{4.2}) in Theorem \ref{theo-4.1}.
\end{rem}

\section{Proof of Theorem \ref{main-theo-2}}
In the section, we take $\e=-1$ in (\ref{2.10}):
\begin{align}\label{5.1}
\~L_k(\k)=\sum_{i=0}^{k}C_k^i(-1)^i p_{2k-2i}(\k), \quad \~N_k(\k)=\sum_{i=0}^{k}C_k^i(-1)^i p_{2k-2i+1}(\k).
\end{align}
It is obvious that if $\k \in  \{ \l\in \R^{n-1} ~|~\l_i\geq 1\}$, then
\begin{align}\label{5.2}
\k^{-1}=(\k_1^{-1},\cdots,\k_{n-1}^{-1})\in  \{ \l \in \R^{n-1} ~|~0<\l_i\leq 1\}.
\end{align}
The following lemma follows from \cite[Remark 4.4]{Ge-Wang-Wu2014}.
\begin{lem}\label{lem-5.1}
If $\k\in \{ \l\in \R^{n-1} ~|~\l_i\geq 1\}$, then
\begin{align}\label{5.3}
(-1)^{k-1}\left[\~N_k(\k^{-1})-p_1(\k^{-1})\~L_{k}(\k^{-1})\right] \leq 0,
\end{align}
Equality holds in (\ref{5.3}) if and only if one of the following two cases holds:
\begin{enumerate}[(i)]
\item $\k_i=\k_j,\quad \forall i,j$;
\item if $k\geq 2$, there exist at most $k-1$ elements with $\k_i>1$, while the remaining elements equal to $1$.
\end{enumerate}
\end{lem}
\begin{proof}
The proof follows from the crucial observations due to Ge-Wang-Wu \cite{Ge-Wang-Wu2014}:

The inequality $\~N_k(\k)-p_1(\k)\~L_{k}(\k)\leq 0$ is equivalent to
\begin{align}\label{5.4}
\sum_{\substack{1\leq i_m \leq n-1 \\ i_j\neq i_l (j\neq l)}}\k_{i_1}(\k_{i_2}\k_{i_3}-1)(\k_{i_4}\k_{i_5}-1)\cdots (\k_{i_{2k-2}}\k_{i_{2k-1}}-1)(\k_{i_{2k}}-\k_{i_{2k+1}})^2 \geq 0,
\end{align}
where the summation takes over all $(2k+1)$-elements permutation of $\{1,\cdots,n-1\}$. Together with (\ref{5.2}), we obtain the desired conclusion.
\end{proof}

As a direct corollary, we have the following inequalities, which will be used in establishing the monotonicity of $Q_k(t)$ along the HMCF.
\begin{cor}\label{cor-2}
If $\k\in \{ \l\in \R^{n-1} ~|~\l_i\geq 1\}$, then
\begin{align}\label{5.5}
\sum_{i=0}^{k}C_k^i(-1)^i\left[p_{n-2-2i}(\k)\(-\frac{p_{n-1}(\k)}{p_{n-2}(\k)}\)+p_{n-1-2i}(\k)\right] \leq 0,
\end{align}
Equality holds in (\ref{5.5}) if and only if one of the following two cases holds:
\begin{enumerate}[(i)]
\item $\k_i=\k_j,\quad \forall i,j$;
\item if $k\geq 2$, there exist at most $k-1$ elements with $\k_i>1$, while the remaining elements equal to $1$.
\end{enumerate}
\end{cor}
\begin{proof}It is obvious that
\begin{align*}
p_j(\k^{-1})=\frac{p_{n-1-j}(\k)}{p_{n-1}(\k)}.
\end{align*}
Thus we have
\begin{align*}
\~L_k(\k^{-1})=\sum_{i=0}^{k}C_k^i(-1)^i p_{2k-2i}(\k^{-1})=\frac{1}{p_{n-1}(\k)}\sum_{i=0}^{k}C_k^i(-1)^i p_{n-1-2k+2i}(\k),
\end{align*}
and the conclusion follows directly from Lemma \ref{lem-5.1}.
\end{proof}

Now we prove the monotonicity of the functional
$$
Q_k(t)=\(\int_{\Sigma}p_{n-1}\)^{-\frac{n-1-2k}{n-1}}\left[\sum_{i=0}^{k}C_k^i(-1)^i\int_{\Sigma} p_{n-1-2i}\right],\quad 0< 2k< n-1
$$
along the HMCF (\ref{3.1}).
\begin{lem}\label{lem-5.2} Let $\Sigma$ be a h-convex hypersurface in $\H^n$. Along the HMCF (\ref{3.1}), the quantity $Q_k(t)$ is monotone decreasing. Moreover, $\frac{d}{dt}Q(t)=0$ at some time $t$ if and only if the equality holds in (\ref{5.5}) everywhere on $\Sigma_t$.
\end{lem}
\begin{proof}
By the variation formula (\ref{4.1}) along the HMCF (\ref{3.1}), we have
\begin{align*}
&\frac{d}{dt}\left[\int_{\Sigma} \sum_{i=0}^{k}C_k^i(-1)^i p_{n-1-2i} \right]\\
=&\int_{\Sigma} \sum_{i=0}^{k} C_k^i(-1)^i\left[2i p_{n-2i}+(n-1-2i)p_{n-2-2i} \right]\(-\frac{p_{n-1}}{p_{n-2}}\)\\
=&\int_{\Sigma} \sum_{i=0}^{k} C_k^i(-1)^i(n-1-2i)p_{n-2-2i}\(-\frac{p_{n-1}}{p_{n-2}}\) \\
 &+\int_{\Sigma} \sum_{j=0}^{k-1} C_k^{j+1} (-1)^{j+1} (2j+2)p_{n-2-2j}\(-\frac{p_{n-1}}{p_{n-2}}\) \\
=&\int_{\Sigma} \sum_{i=0}^{k} C_k^i(-1)^i(n-1-2k)p_{n-2-2i}\(-\frac{p_{n-1}}{p_{n-2}}\) \\
 &-\int_{\Sigma} \sum_{j=0}^{k-1} 2(-1)^{j+1} \left[C_k^j (k-j)-C_k^{j+1}(j+1)\right]p_{n-2-2j}\(-\frac{p_{n-1}}{p_{n-2}}\).
\end{align*}
Together with the identities
$$
C_k^j (k-j)-C_k^{j+1}(j+1)=0, \quad 0\leq j \leq k-1,
$$
we obtain
\begin{equation}\label{5.6}
\frac{d}{dt}\left[\int_{\Sigma} \sum_{i=0}^{k}C_k^i(-1)^i p_{n-1-2i} \right]=(n-1-2k)\int_{\Sigma} \sum_{i=0}^{k} C_k^i(-1)^ip_{n-2-2i}\(-\frac{p_{n-1}}{p_{n-2}}\).
\end{equation}
As the initial hypersurface $\Sigma$ is h-convex, it follows from Lemma \ref{lem-3.9} that $\Sigma_t$ is strictly h-convex for $t\in (0,T^\ast)$.
Together with (\ref{5.5}) we deduce that
\begin{align}\label{5.7}
\frac{d}{dt}\left[\int_{\Sigma} \sum_{i=0}^{k}C_k^i(-1)^i p_{n-1-2i} \right]\leq -(n-1-2k)\left[ \int_{\Sigma} \sum_{i=0}^{k}C_k^i(-1)^i p_{n-1-2i} \right].
\end{align}
The inequalities (\ref{4.4}) and (\ref{5.7}) imply that $\frac{d}{dt}Q_k(t)\leq 0$. If $\frac{d}{dt}Q_k(t)=0$ at some time $t$, then the equality holds in (\ref{5.5}) everywhere on $\Sigma_t$.
\end{proof}

Now we are in a position to prove Theorem \ref{main-theo-2}.
\begin{proof}[Proof of Theorem \ref{main-theo-2}]
Let $\Sigma_t$ be the evolving hypersurface for $t\in [0,T)$ along the HMCF (\ref{3.1}), and denote by $\Omega_t$ the domain enclosed by $\Sigma_t$. It follows from Lemma \ref{lem-3.6} that
\begin{align*}
\lim_{t\ra T^\ast}Q_k(t)=\lim_{t\ra T^\ast}\left[\(\int_{\Sigma}p_{n-1}\)^{-\frac{n-1-2k}{n-1}}\int_{\Sigma} \sum_{i=0}^{k}C_k^i(-1)^i p_{n-1-2i}\right]=\omega_{n-1}^{\frac{2k}{n-1}}.
\end{align*}
By Lemma \ref{lem-5.2}, $Q_k(t)$ is monotone decreasing along the HMCF (\ref{3.1}). We obtain
\begin{align*}
Q_k(0) \geq Q_k(t) \geq \lim_{t\ra T^\ast}Q_k(t)=\omega_{n-1}^{\frac{2k}{n-1}}.
\end{align*}
Together with (\ref{4.7}), we deduce that
\begin{align}\label{5.8}
\sum_{i=0}^{k}C_k^i (-1)^i\int_{\Sigma} p_{n-1-2i}\geq \omega_{n-1}^{\frac{2k}{n-1}}\(\int_{\Sigma}p_{n-1}\)^{\frac{n-1-2k}{n-1}}.
\end{align}

If the equality holds in (\ref{5.8}), then $Q_k(t)\equiv \omega_{n-1}^{\frac{2k}{n-1}}$ for all $t\in[0,T^\ast)$. By Lemma \ref{lem-5.2}, the equality holds in (\ref{5.5}) everywhere on $\Sigma_t$, $t\in [0,T^\ast)$. By Lemma \ref{lem-3.9}, the evolving hypersurface $\Sigma_t$ is strictly h-convex for $t\in (0,T^\ast)$, which excludes the case (ii) in Corollary \ref{cor-2}. Thus, we conclude that $\Sigma_t$ is totally umbilical and hence a geodesic sphere in $\H^n$ for $t\in (0,T^\ast)$. As $t\ra 0$, the initial hypersurface $\Sigma$ is smoothly approximated by a family of geodesic spheres, and thus it is also a geodesic sphere in $\H^n$. It is easy to see that if $\Sigma$ is a geodesic sphere of radius $\rho$, then $p_k=\coth^k \rho$ and
\begin{align*}
\sum_{i=0}^{k}C_k^i (-1)^i\int_{\Sigma} p_{n-1-2i}=&\sum_{i=0}^{k}C_k^i (-1)^i \omega_{n-1} \coth^{n-1-2i} \rho\sinh^{n-1} \rho\\
=&\omega_{n-1}\cosh^{n-1}\rho\sum_{i=0}^{k}C_k^i (-1)^i (\tanh^2 \rho)^i \\
=&\omega_{n-1}\cosh^{n-1-2k}\rho,
\end{align*}
and
\begin{align*}
\omega_{n-1}^{\frac{2k}{n-1}}\(\int_{\Sigma}p_{n-1}\)^{\frac{n-1-2k}{n-1}}=\omega_{n-1}\cosh^{n-1-2k}\rho.
\end{align*}
Thus, the equality holds in (\ref{5.8}) on a geodesic sphere. This completes the proof of Theorem \ref{main-theo-2}.
\end{proof}

\section{Proof of Theorem \ref{main-theo-3}}
In this section, we will use the inverse mean curvature flow and Theorem \ref{main-theo-1} to give the proof of Theorem \ref{main-theo-3}. Let $X_0:M^{n-1}\ra \H^n$ be a smooth embedding such that $\Sigma=X_0(M)$ is a closed smooth hypersurface in $\H^n$. We consider the smooth family of immersions $X:M \times [0,T)\ra \H^{n}$ evolving along the inverse mean curvature flow (IMCF):
\begin{align}\label{6.1}\left\{ \begin{aligned}
\frac{\partial}{\partial t}X(x,t)=&\frac{1}{H(x,t)}\nu(x,t),\\
X(\cdot,0)=&X_0(\cdot),\end{aligned}\right.
\end{align}
where $H(x,t)$ is the mean curvature and $\nu(x,t)$ is the unit outward normal vector of the hypersurface $\Sigma_t=X(M,t)$, respectively. Along the IMCF (\ref{6.1}), we have the following evolution equations on the Weingarten tensor $\mathcal{W}=(h_i^j)$ of $\Sigma_t$ (see e.g. \cite{Hu-Li2018}):
\begin{align}\label{6.2}
\frac{\partial}{\partial t}h_i^j=\frac{1}{H^2}\D h_{i}^{j}-\frac{2}{H^3}\nabla_i H\nabla^j H+\frac{1}{H^2}(|A|^2+n-1)h_i^j-\frac{2}{H}(h^2)_i^j.
\end{align}
By the variational formula \cite{Reilly1973}, one can check that along the IMCF (\ref{6.1}) we have
\begin{align}\label{6.3}
\frac{d}{dt}\int_{\Sigma}p_k=\int_{\Sigma} \( (n-1-k)p_{k+1}+k p_{k-1}\)\frac{1}{(n-1)p_1}, \quad k=0,\cdots,n-1.
\end{align}

In a recent work \cite{Hu-Li2018}, the second and third authors showed that the nonnegativity of the sectional curvature of the evolving hypersurfaces is preserved along the IMCF (\ref{6.1}). The proof relies on a crucial property that the strict convexity of the evolving hypersurface is also preserved along the IMCF. This property can also be deduced directly from the pinching estimate, see \cite[Lemma 3.2]{Wei2018}.
\begin{lem}\label{lem-6.1}
Let $\Sigma_t$, $t\in [0,T)$ be a solution of the IMCF (\ref{6.1}) in $\H^n$. If the initial hypersurface $\Sigma$ is strictly convex, then the evolving hypersurface $\Sigma_t$ is strictly convex for $t\in (0,T)$.
\end{lem}

By using the result of Gerhardt \cite{Gerhardt2011} and Lemma \ref{lem-6.1}, we have the following proposition.
\begin{pro}\label{proposition-6.2}
If the initial hypersurface $\Sigma$ is strictly convex, then the solution of the IMCF (\ref{6.1}) exists for all $t>0$ and preserves the strict convexity. Moreover, the hypersurfaces $\Sigma_t$ become more and more umbilical in the following sense:
\begin{align}\label{6.5}
\left| h_i^j-\d_i^j \right| \leq C e^{-\frac{t}{n-1}}, \quad t>0,
\end{align}
i.e., the principal curvatures are uniformly bounded and converge exponentially fast to $1$.
\end{pro}

Now we give the proof of Theorem \ref{main-theo-3}.
\begin{proof}[Proof of Theorem \ref{main-theo-3}]
By the variation formula (\ref{6.3}) along the IMCF (\ref{6.1}) and the Newton-MacLaurin inequality (\ref{2.1}), we have
\begin{align}\label{6.6}
\frac{d}{dt}\(\frac{|\Sigma_t|}{\omega_{n-1}}\)=&\frac{|\Sigma_t|}{\omega_{n-1}},
\end{align}
\begin{align}\label{6.7}
\frac{d}{dt}\(\frac{\int_{\Sigma}p_{n-1}}{\omega_{n-1}}\)=&\frac{1}{\omega_{n-1}}\int_\Sigma \frac{p_{n-2}}{p_{1}}\leq \frac{\int_{\Sigma}p_{n-3}}{\omega_{n-1}}.
\end{align}
By Lemma \ref{lem-6.1}, the evolving hypersurface $\Sigma_t$ is strictly convex. By (\ref{1.7}) in Theorem \ref{main-theo-1}, we have
\begin{align*}
\frac{\int_{\Sigma}p_{n-3}}{\omega_{n-1}} \leq \frac{\int_{\Sigma}p_{n-1}}{\omega_{n-1}}\left[1-\(\frac{\int_\Sigma p_{n-1}}{\omega_{n-1}}\)^{-\frac{2}{n-1}} \right].
\end{align*}
For simplicity, we take
$$
x(t)=\frac{\int_{\Sigma}p_{n-1}}{\omega_{n-1}}, \quad y(t)=\frac{|\Sigma_t|}{\omega_{n-1}}.
$$
It follows from (\ref{6.6}) and (\ref{6.7}) that
\begin{align*}
\frac{d}{dt}x \leq x \(1-x^{-\frac{2}{n-1}}\), \quad \frac{d}{dt}y=y.
\end{align*}
A direct calculation gives
\begin{align*}
 &\frac{d}{dt} \left[ y-x (1-x^{-\frac{2}{n-1}})^\frac{n-1}{2}\right]\\
=& y-\left[(1-x^{-\frac{2}{n-1}})^\frac{n-1}{2}+x^{-\frac{2}{n-1}}(1-x^{-\frac{2}{n-1}})^\frac{n-3}{2}\right]\frac{d}{dt}x\\
=&y-\left[(1-x^{-\frac{2}{n-1}})^\frac{n-3}{2}\right]\frac{d}{dt}x \\
\geq &y-x (1-x^{-\frac{2}{n-1}})^\frac{n-1}{2}.
\end{align*}
Along the IMCF (\ref{6.1}), we have $y(t)>0$ for all $t\in [0,\infty)$. For any bounded convex domain $\Omega_t$ in $\H^n$, $W_n(\Omega_t)=\frac{\omega_{n-1}}{n}$. Together with (\ref{2.6}), we get
\begin{align*}
\(\frac{\int_{\Sigma}p_{n-1}}{\omega_{n-1}}\)=C_n^2\(\frac{W_{n-2}(\Omega)}{\omega_{n-1}}\)+1\geq 1.
\end{align*}
We consider the following functional
\begin{align*}
Q(t):=\(\frac{|\Sigma_t|}{\omega_{n-1}}\)^{-1}\left[\frac{|\Sigma_t|}{\omega_{n-1}}-\(\frac{\int_{\Sigma_t}p_{n-1}}{\omega_{n-1}}\)\left[1-\(\frac{\int_{\Sigma_t}p_{n-1}}{\omega_{n-1}}\)^{-\frac{2}{n-1}}\right]^\frac{n-1}{2}\right].
\end{align*}
Then we have $\frac{d}{dt}Q(t)\geq 0$. Now we analyze the asymptotics of $Q(t)$ as $t\ra \infty$. By (\ref{6.6}) we have $|\Sigma_t|=|\Sigma|e^t$. By (\ref{6.5}) in Proposition \ref{proposition-6.2}, we get
$$
h_i^j=\(1+O(e^{-\frac{t}{n-1}})\)\d_i^j, \quad \text{on $\Sigma_t$}.
$$
As $p_{n-1}$ is homogeneous of degree $n-1$, we get
\begin{align*}
p_{n-1}(h_i^j)=(1+O(e^{-\frac{t}{n-1}}))^{n-1}=1+O(e^{-\frac{t}{n-1}}), \quad \text{on $\Sigma_t$}.
\end{align*}
Thus, we have
\begin{align*}
\frac{\int_{\Sigma_t}p_{n-1}}{\omega_{n-1}}=\frac{|\Sigma_t|}{\omega_{n-1}}\(1+O(e^{-\frac{t}{n-1}})\)=O(e^{t}),  \quad \text{on $\Sigma_t$}.
\end{align*}
It follows that
\begin{align*}
Q(t)=&1-\(\frac{\int_{\Sigma_t}p_{n-1}}{|\Sigma_t|}\)\left[1-\(\frac{\int_{\Sigma_t}p_{n-1}}{\omega_{n-1}}\)^{-\frac{2}{n-1}}\right]^\frac{n-1}{2}\\
=&1-\(1+O(e^{-\frac{t}{n-1}})\)\(1+O(e^{-\frac{2t}{n-1}})\)^\frac{n-1}{2}\\
=&1-\(1+O(e^{-\frac{t}{n-1}})\)\(1+O(e^{-\frac{2t}{n-1}})\)\\
=&O(e^{-\frac{t}{n-1}}),
\end{align*}
which gives $\lim_{t\ra \infty}Q(t)=0$. Together with the monotonicity of $Q(t)$, we obtain
$$
Q(0) \leq Q(t) \leq \lim_{t\ra \infty}Q(t)=0,
$$
which is equivalent to (\ref{1.11}). If the equality holds in (\ref{1.11}), then the equality in the Newton-MacLaurin inequality implies that $\Sigma$ is totally umbilical and hence it is a geodesic sphere in $\H^n$.
\end{proof}

\section{Alexandrov-Fenchel inequalities in Euclidean space via curvature contraction flow}
In this section, we use curvature contraction flows to prove Alexandrov-Fenchel inequalities for convex hypersurfaces in Euclidean space. Let $\Sigma^{n-1}$ be a strictly convex hypersurface in $\R^n$. Let $X_0:M^{n-1}\ra \R^n$ is the embedding of the hypersurface $\Sigma$ in $\R^n$, we consider a family of smooth immersions $X:M^{n-1}\times [0,T^\ast)\ra \R^n$ satisfying
\begin{equation}\label{7.1}
\begin{split}
\left\{\begin{aligned}
\frac{\partial}{\partial t} X(x,t)=&-\frac{p_{n-k}}{p_{n-k-1}}\nu(x,t), \quad 1\leq k\leq n-1,\\
X(0,\cdot)=&X_0(\cdot).
\end{aligned}\right.
\end{split}
\end{equation}
The smooth convergence of these curvature contraction flows (\ref{7.1}) in Euclidean space have been proved by the first author \cite{Andrews1994-1}.

\begin{pro}\label{pro-convergence}
If $\Sigma$ is a strictly convex hypersurface in $\R^n$, then there exists a unique smooth solution to the flow (\ref{7.1}) on a maximal time interval $[0,T^\ast)$, and the hypersurfaces $\Sigma_t$ converge uniformly to a round point $p_0\in \R^n$ as $t\ra T^\ast$, in the sense that the rescaled flow converges smoothly to a round sphere. Moreover, the flow hypersurface $\Sigma_t$ is strictly convex for each $t\in[0,T^\ast)$.
\end{pro}

The variation formula of $\int_{\Sigma}p_m$ along (\ref{7.1}) is
\begin{align}\label{7.2}
\frac{d}{dt}\int_{\Sigma} p_{m}=-(n-1-m)\int_{\Sigma}\frac{p_{n-k}p_{m+1}}{p_{n-k-1}}, \quad m=0,1,\cdots,n-1.
\end{align}
For a convex domain $\Omega$ in $\R^n$ with smooth boundary $\partial\Omega=\Sigma$, the quermassintegrals and the curvature integrals are related by
\begin{align}\label{7.3}
\int_{\Sigma}p_j =n W_{j+1}(\Omega), \quad j=0,1,\cdots,n-1.
\end{align}
It follows from the Gauss-Bonnet-Chern theorem that the total curvature is a constant, i.e.,
\begin{align}\label{7.4}
\int_{\Sigma}p_{n-1}=\omega_{n-1}.
\end{align}
A similar argument as in Lemma \ref{lem-3.6} gives the following result.
\begin{lem}\label{lem-7.2}
Let $\Sigma$ be a strictly convex hypersurface in $\R^n$. Let $\Sigma_t$, $t\in [0,T^\ast)$ be the solution of the flow (\ref{7.1}) with the initial hypersurface $\Sigma$, and $\Omega_t$ be the domain enclosed by $\Sigma_t$.  Then we have
\begin{equation}\label{7.5}
\begin{split}
\lim_{t\ra T^\ast}\int_{\Sigma_t}p_j=\left\{\begin{aligned}
&0, \quad    &0\leq j\leq n-2;\\
&\omega_{n-1}, \quad &j=n-1,
\end{aligned}\right.
\end{split}
\end{equation}
\end{lem}

The following Alexandrov-Fenchel inequality has been proved for $k$-convex and starshaped hypersurfaces in $\R^n$ by Guan and Li \cite{Guan-Li2009}. We show that this inequality can also be proved for strictly convex hypersurfaces via curvature contraction flows in $\R^n$.
\begin{theo}\label{theo-7.1}
For any strictly convex hypersurface $\Sigma$ in $\R^n$, we have the following inequality
\begin{align}\label{7.6}
\(\frac{\int_{\Sigma}p_{n-1-k}}{\omega_{n-1}}\)^\frac{1}{k}\geq \(\frac{\int_{\Sigma}p_{n-1-(k+1)}}{\omega_{n-1}}\)^\frac{1}{k+1},\quad 1\leq k\leq n-1.
\end{align}
The equality holds if and only $\Sigma$ is a geodesic sphere.
\end{theo}
\begin{proof}
We first prove (\ref{7.6}) for $k=1$. By (\ref{7.2}), we have
\begin{align*}
\frac{d}{dt}\int_{\Sigma} p_{n-3}=-2\int_{\Sigma} p_{n-1}=-2\omega_{n-1},
\end{align*}
and
\begin{align}\label{7.7}
\frac{d}{dt}\int_{\Sigma} p_{n-2}=-\int_{\Sigma}\frac{p_{n-1}^2}{p_{n-2}}.
\end{align}
By the H\"older inequality, we have
\begin{align}\label{7.8}
\(\int_{\Sigma}p_{n-1}\)^2 =\(\int_{\Sigma}\frac{p_{n-1}}{\sqrt{p_{n-2}}}\sqrt{p_{n-2}}\)^2 \leq \int_{\Sigma}\frac{p_{n-1}^2}{p_{n-2}} \cdot \int_{\Sigma}p_{n-2}.
\end{align}
Together with (\ref{7.7}), we get
\begin{align*}
\frac{d}{dt}\(\int_{\Sigma}p_{n-2}\)^2\leq -2\(\int_{\Sigma}p_{n-1}\)^2=-2\omega_{n-1}^2.
\end{align*}
We define the functional
$$
Q_1(t):=\(\int_{\Sigma}p_{n-2}\)^2-\omega_{n-1}\int_{\Sigma}p_{n-3},
$$
then $\frac{d}{dt}Q_1(t)\leq 0$. Together with Lemma \ref{7.2}, we get
\begin{align*}
\(\int_{\Sigma}p_{n-2}\)^2-\omega_{n-1}\int_{\Sigma}p_{n-3}=Q_1(0)\geq \lim_{t\ra T^\ast} Q_1(t)=0,
\end{align*}
which verifies (\ref{7.6}) for $k=1$.

Now we prove (\ref{7.6}) for $k\geq 2$ by induction. We claim that if (\ref{7.6}) holds for $k-1$, i.e.,
\begin{align}\label{7.9}
\(\frac{\int_{\Sigma}p_{n-1-(k-1)}}{\omega_{n-1}}\)^\frac{1}{k-1} \geq \(\frac{\int_{\Sigma}p_{n-1-k}}{\omega_{n-1}}\)^\frac{1}{k}.
\end{align}
then (\ref{7.6}) also holds for $k$, i.e.,
\begin{align}\label{7.10}
\(\frac{\int_{\Sigma}p_{n-1-k}}{\omega_{n-1}}\)^\frac{1}{k} \geq \(\frac{\int_{\Sigma}p_{n-1-(k+1)}}{\omega_{n-1}}\)^\frac{1}{k+1}.
\end{align}
By (\ref{7.2}), we have
\begin{align*}
\frac{d}{dt}\int_{\Sigma}p_{n-1-(k+1)}=-(k+1)\int_{\Sigma}p_{n-k}.
\end{align*}
\begin{align}\label{7.11}
\frac{d}{dt}\int_{\Sigma}p_{n-1-k}=-k\int_{\Sigma}\frac{p_{n-k}^2}{p_{n-1-k}}\leq -k\frac{\(\int_{\Sigma}p_{n-k}\)^2}{\int_{\Sigma}p_{n-1-k}},
\end{align}
where we have used the H\"older inequality in the last inequality. We define the functional
\begin{align*}
Q_k(t):=\(\int_{\Sigma} p_{n-1-k}\)^\frac{k+1}{k}-\omega_{n-1}^{\frac{1}{k}}\int_{\Sigma}p_{n-1-(k+1)}.
\end{align*}
Then we get
\begin{align*}
\frac{d}{dt}Q_k(t)=&-(k+1)\(\int_{\Sigma}p_{n-1-k}\)^\frac{1}{k}\int_{\Sigma}\frac{p_{n-k}^2}{p_{n-1-k}}+(k+1)
\omega_{n-1}^\frac{1}{k}\int_{\Sigma}p_{n-k}\\
\leq &-(k+1)\frac{\(\int_{\Sigma}p_{n-k}\)^2}{\(\int_{\Sigma}p_{n-1-k}\)^{\frac{k-1}{k}}}+(k+1)\omega_{n-1}^\frac{1}{k}\int_{\Sigma}p_{n-k}\\
=&(k+1)\frac{\int_{\Sigma}p_{n-k}}{\(\int_{\Sigma}p_{n-1-k}\)^{\frac{k-1}{k}}}\left[-\int_{\Sigma}p_{n-k}+\omega_{n-1}^{\frac{1}{k}}\(\int_{\Sigma}p_{n-1-k}\)^{\frac{k-1}{k}} \right].
\end{align*}
It is easy to observe that (\ref{7.9}) is equivalent to
\begin{align*}
-\int_{\Sigma}p_{n-k}+\omega_{n-1}^{\frac{1}{k}}\(\int_{\Sigma}p_{n-1-k}\)^{\frac{k-1}{k}}\leq 0,
\end{align*}
which implies that $\frac{d}{dt}Q_k(t)\leq 0$. Together with Lemma \ref{7.2}, we get
\begin{align*}
\(\int_{\Sigma} p_{n-1-k}\)^\frac{k+1}{k}-\omega_{n-1}^{\frac{1}{k}}\int_{\Sigma}p_{n-1-(k+1)}=Q_k(0) \geq \lim_{t\ra T^\ast}Q_k(t)=0,
\end{align*}
which is equivalent to (\ref{7.10}). To characterize the equality case of (\ref{7.6}), the equality holds in (\ref{7.8}) and (\ref{7.11}). Together with the strict convexity of $\Sigma$, we have the curvature quotient is constant, i.e., $\frac{p_{n-k}}{p_{n-1-k}}\equiv C$ for some constant $C>0$ on $\Sigma$. By a rigidity theorem of Korevaar \cite{Korevaar1988}, we conclude that $\Sigma$ is a geodesic sphere in $\R^n$. This completes the proof of Theorem \ref{theo-7.1}.
\end{proof}

\section{A Heintze-Karcher type inequality for hypersurface with positive Ricci curvature in hyperbolic space}
For a compact domain $\Omega$ with smooth boundary $\partial \Omega=\g$ in hyperbolic plane $\H^2$, the Gauss-Bonnet theorem implies
\begin{align*}
\int_{\g}\k ds = 2\pi+V,
\end{align*}
where $\k$ is the curvature of $\g$ and $V$ is the volume of $\Omega$. Let $L$ be the length of $\g$, then the isoperimetric inequality gives
\begin{align*}
L^2 \geq V(4\pi+V).
\end{align*}
Therefore, for a strictly convex curve $\g$ in $\H^2$, we have the following inequality:
\begin{align}\label{8.1}
\int_{\g} \frac{1}{\k}ds \geq \frac{L^2}{\int_{\g} \k ds} \geq \frac{V(4\pi+V)}{2\pi+V}.
\end{align}

In this section, we will use the mean curvature flow to prove the following Heintze-Karcher type inequality for hypersurfaces with positive Ricci curvature in hyperbolic space $\H^n$.
\begin{theo}\label{main-theo-4}
Let $\Sigma$ be a closed hypersurface with positive Ricci curvature in $\H^n$, and $\Omega$ the domain enclosed by $\Sigma$. Then
\begin{align}\label{8.2}
\int_{\Sigma} \frac{1}{p_1} \geq \frac{f^{-1}(\mrm{Vol}(\Omega))}{\sqrt{1+\(\frac{f^{-1}(\Vol(\Omega))}{\omega_{n-1}}\)^{-\frac{2}{n-1}}}}.
\end{align}
Here $f=f(x)$ is a strictly increasing function defined by
$$
f(x)=\int_{0}^{\sinh^{-1}\left[\(\frac{x}{\omega_{n-1}}\)^\frac{1}{n-1}\right]}\omega_{n-1}\sinh^{n-1}(s)ds.
$$
Moreover, the equality holds if and only if $\Sigma$ is a geodesic sphere in $\H^n$.
\end{theo}
\begin{rem}
When $n=2$, (\ref{8.2}) reduces to (\ref{8.1}).
\end{rem}
Given a smooth compact immersion $X_0:M^{n-1} \ra \H^n$, where $n\geq 3$, we consider a smooth family of immersion $X:M^{n-1}\times [0,T) \ra \H^n$ with the initial data $X_0$ which evolves by
\begin{align}\label{8.3}
\left\{\begin{aligned}
&\frac{\partial X}{\partial t}=-H\nu,\\
&X(\cdot,0)=X_0,
\end{aligned}\right.
\end{align}
where $\nu$ is the outward unit normal and $H$ is the mean curvature of the flow hypersurface $\Sigma_t=X(M,t)$, respectively. This flow is called the {\em mean curvature flow} (briefly, MCF).

In the landmark work \cite{Huisken1984}, Huisken proved that for a compact convex hypersurface in $\R^n$, it smoothly evolves along the MCF until it shrinks to a round point.
For hypersurfaces in non-Euclidean background spaces, the understanding of behavior is less complete, see \cite{Andrews2003,Huisken1986}.
When the ambient space is the hyperbolic space $\H^n$, Huisken \cite{Huisken1986} proved the smooth convergence of the MCF under the condition $\k_i H \geq n-1$ for all $i$, which is weaker than horospherical convexity (i.e., all principal curvatures $\k_i\geq 1$).
Recently, Chen and the first author \cite[Theorem 1]{Andrews-Chen2015} proved the smooth convergence of the MCF in hyperbolic space under the condition of positive Ricci curvature (i.e., $\k_i(H-\k_i)>n-2$ for all $i$).

\begin{theo}\cite{Andrews-Chen2015}\label{converegence-theorem}
For any embedding $X_0:M^{n-1} \ra \H^{n}$ with positive Ricci curvature, there exists a smooth solution of the MCF (\ref{8.3}) on a maximal time interval $[0,T^\ast)$. The hypersurfaces $\Sigma_t=X_t(M)$ have positive Ricci curvature for each $t\in (0,T^\ast)$, and smoothly converge to a round point as $t\ra T^\ast$.
\end{theo}

We first collect the following evolution equations along the MCF (\ref{8.1}), see \cite{Huisken1984}.
\begin{lem}
Along the MCF (\ref{8.1}), we have the following evolution equations:
\begin{align}\label{8.4}
\frac{\partial}{\partial t} H=\D H+(|A|^2-(n-1))H.
\end{align}
\begin{align}\label{8.5}
\frac{\partial}{\partial t} d\mu_t=-H^2 d\mu_t.
\end{align}
\end{lem}

\begin{lem}
Along the MCF (\ref{8.1}), we have the following variational formulas:
\begin{align}\label{8.6}
\frac{d}{dt}\mrm{Vol}(\Omega_t)=-\int_{\Sigma_t} (n-1)p_1.
\end{align}
\begin{align}\label{8.7}
\frac{d}{dt}\int_{\Sigma_t}\frac{1}{p_1}=\int_{\Sigma_t} \left[-\frac{2}{p_1^3}|\nabla p_1|^2-\frac{\(|A|^2-(n-1)\)}{p_1}-(n-1)^2 p_1 \right].
\end{align}
\end{lem}
\begin{proof}
For a proof of (\ref{8.6}), see e.g. \cite{Reilly1973}. We give a proof of (\ref{8.7}).
\begin{align*}
\frac{d}{dt}\int_{\Sigma_t}\frac{1}{p_1}d\mu_t=& -\int_{\Sigma_t}\frac{1}{p_1^2}\frac{\partial}{\partial t}p_1 d\mu_t-(n-1)^2\int_{\Sigma_t} p_1 d\mu_t \\
=&-\int_{\Sigma_t}\frac{1}{p_1^2}\(\D p_1 +\(|A|^2-(n-1)\)p_1\)d\mu_t-(n-1)^2\int_{\Sigma_t} p_1 d\mu_t \\
=&\int_{\Sigma_t}\left[-\frac{2}{p_1^3}|\nabla p_1|^2-\frac{\(|A|^2-(n-1)\)}{p_1}-(n-1)^2 p_1 \right]d\mu_t.
\end{align*}
\end{proof}

If the initial hypersurface $\Sigma$ has positive Ricci curvature, by Theorem \ref{converegence-theorem} we know that the flow hypersurface $\Sigma_t$ of the MCF (\ref{8.1}) has positive Ricci curvature for $t\in (0,T^\ast)$, where $T^\ast$ is the maximal existence time. Moreover, the principal curvature is pinched and hence the inner radius and outer radius is comparable as $t\ra T^\ast$. The following lemma plays a key role in our proof of Theorem \ref{main-theo-4}.
\begin{lem}\label{key-lem-1}
Let $\Sigma$ be a hypersurface with positive Ricci curvature in $\H^n$. Let $\Sigma_t$, $t\in [0,T^\ast)$ be the solution of the MCF (\ref{8.1}) with the initial hypersurface $\Sigma$, and $\Omega_t$ be the domain enclosed by $\Sigma_t$, then
\begin{align*}
\lim_{t\ra T^\ast}W_k(\Omega_t)=\left\{\begin{aligned}
&0, \quad    &0\leq k\leq n-1;\\
&\frac{\omega_{n-1}}{n}, \quad &k=n,
\end{aligned}\right.
\end{align*}
and
\begin{align*}
\lim_{t\ra T^\ast}\int_{\Sigma_t}p_j=\left\{\begin{aligned}
&0, \quad    &0\leq j\leq n-2;\\
&\omega_{n-1}, \quad &j=n-1,
\end{aligned}\right.
\end{align*}
\end{lem}
\begin{proof}
The proof is similar to the proof of Lemma \ref{lem-3.6}, since we have the pinching estimate along the MCF, see \cite[Corollary 10]{Andrews-Chen2015}.
\end{proof}

Here we establish the monotonicity of the following functional
\begin{align*}
Q(t):=e^{-(n-1)t}\left[\int_{\Sigma_t} \frac{1}{p_1}-\frac{f^{-1}(\mrm{Vol}(\Omega_t))}{\sqrt{1+\(\frac{f^{-1}(\Vol(\Omega_t))}{\omega_{n-1}}\)^{-\frac{2}{n-1}}}}\right].
\end{align*}

\begin{lem}\label{lem-monotonicity}
Along the MCF (\ref{8.1}), $Q(t)$ is monotone decreasing unless it is totally umbilical.
\end{lem}
\begin{proof}
By (\ref{8.7}) and the Newton-MacLaurin inequality, we have
\begin{align*}
\frac{d}{dt}\int_{\Sigma_t} \frac{1}{p_1}\leq &\int_{\Sigma_t} \left[-(n-1)\(p_1-\frac{1}{p_1}\)-(n-1)^2 p_1 \right] \\
= &\int_{\Sigma_t} \left[(n-1)\frac{1}{p_1}-n(n-1)p_1 \right].
\end{align*}
We also have
\begin{align*}
 \frac{d}{dt} f^{-1}(\mrm{Vol}(\Omega_t))=&\frac{1}{f'(f^{-1}(\mrm{Vol}(\Omega_t)))} \frac{d}{dt} \mrm{Vol}(\Omega_t) \\
=&-(n-1)^2 \sqrt{1+\(\frac{f^{-1}(\mrm{Vol}(\Omega_t))}{\omega_{n-1}}\)^{-\frac{2}{n-1}}} \int_{\Sigma_t}p_1,
\end{align*}
and hence
\begin{align*}
 \frac{d}{dt}\left[ \frac{f^{-1}(\mrm{Vol}(\Omega_t))}{\sqrt{1+\(\frac{f^{-1}(\Vol(\Omega_t))}{\omega_{n-1}}\)^{-\frac{2}{n-1}}}}\right]=&\frac{1+\frac{n}{n-1}\(\frac{f^{-1}(\mrm{Vol}(\Omega_t))}{\omega_{n-1}}\)^{-\frac{2}{n-1}}}{\(1+\(\frac{f^{-1}(\mrm{Vol}(\Omega_t))}{\omega_{n-1}}\)^{-\frac{2}{n-1}}\)^\frac{3}{2}}\frac{d}{dt} f^{-1}(\mrm{Vol}(\Omega_t))\\
=&-(n-1)^2\frac{1+\frac{n}{n-1}\(\frac{f^{-1}(\mrm{Vol}(\Omega_t))}{\omega_{n-1}}\)^{-\frac{2}{n-1}}}{1+\(\frac{f^{-1}(\mrm{Vol}(\Omega_t))}{\omega_{n-1}}\)^{-\frac{2}{n-1}}} \int_{\Sigma_t}p_1.
\end{align*}
Combining these estimates together, we get
\begin{equation}\label{7.3.2}
\begin{split}
 &\frac{d}{dt}\left[\int_{\Sigma_t} \frac{1}{p_1}-\frac{f^{-1}(\mrm{Vol}(\Omega_t))}{\sqrt{1+\(\frac{f^{-1}(\Vol(\Omega_t))}{\omega_{n-1}}\)^{-\frac{2}{n-1}}}}\right]\\
\leq &(n-1)\int_{\Sigma_t}\frac{1}{p_1} -(n-1)\frac{\int_{\Sigma_t}p_1 }{1+\(\frac{f^{-1}(\Vol(\Omega_t))}{\omega_{n-1}}\)^{-\frac{2}{n-1}}}.
\end{split}
\end{equation}
On the other hand, the evolving hypersurfaces $\Sigma_t$ have positive Ricci curvature, so we have
\begin{align*}
\int_{\Sigma_t} p_1 \geq |\Sigma_t| \left[ 1+\(\frac{|\Sigma_t|}{\omega_{n-1}}\)^{-\frac{2}{n-1}}\right]^{\frac{1}{2}}.
\end{align*}
We define the function $h:[0,\infty)\ra \R_{+}$ as
$$
h(x):=x\left[1+\(\frac{x}{\omega_{n-1}}\)^{-\frac{2}{n-1}}\right]^\frac{1}{2}.
$$
It is easy to verify that $h$ is strictly increasing. By the isoperimetric inequality for bounded domains in hyperbolic space \cite{Schmidt1940}, we have
$$
|\Sigma_t| \geq f^{-1}(\mrm{Vol}(\Omega_t)),
$$
and hence
\begin{align*}
\int_{\Sigma_t} p_1 \geq f^{-1}(\mrm{Vol}(\Omega_t)) \left[ 1+\(\frac{f^{-1}(\mrm{Vol}(\Omega_t))}{\omega_{n-1}}\)^{-\frac{2}{n-1}}\right]^{\frac{1}{2}}.
\end{align*}
Together with (\ref{7.3.2}), we get
\begin{align*}
&\frac{d}{dt}\left[\int_{\Sigma_t} \frac{1}{p_1}-\frac{f^{-1}(\mrm{Vol}(\Omega_t))}{\sqrt{1+\(\frac{f^{-1}(\Vol(\Omega_t))}{\omega_{n-1}}\)^{-\frac{2}{n-1}}}}\right] \\
\leq &(n-1)\left[\int_{\Sigma_t} \frac{1}{p_1}-\frac{f^{-1}(\mrm{Vol}(\Omega_t))}{\sqrt{1+\(\frac{f^{-1}(\Vol(\Omega_t))}{\omega_{n-1}}\)^{-\frac{2}{n-1}}}}\right].
\end{align*}
Finally, we obtain $\frac{d}{dt}Q(t) \leq 0$. If $\frac{d}{dt}Q(t) = 0$, then the equality in Newton-MacLaurin inequality implies that $\Sigma_t$ is totally umbilical.
\end{proof}

By the pinching estimate along the MCF (\ref{8.1}), we have $\frac{1}{p_1}\ra 0$ as $t\ra T^\ast$. By Lemma \ref{key-lem-1}, we also have
\begin{align*}
|\Sigma_t| \ra 0, \quad \mrm{Vol}(\Omega_t) \ra 0, \quad \text{as $t\ra T^\ast$}.
\end{align*}
Note that $f^{-1}(0)=0$, we obtain
$$
\lim_{t\ra T^\ast} Q(t)=0.
$$
Finally, combining with the monotonicity of the functional $Q(t)$, we have
\begin{align*}
Q(0) \geq Q(t) \geq \lim_{t\ra T^\ast}Q(t)=0,
\end{align*}
If the equality holds in (\ref{8.2}), we have $Q(t)\equiv 0$ for all $t\in [0,T^\ast)$. By Lemma \ref{lem-monotonicity}, the initial hypersurface $\Sigma$ is totally umbilical. Therefore, $\Sigma$ is a geodesic sphere in $\H^n$. On the other hand, if $\Sigma$ is a geodesic sphere of radius $r$ in $\H^n$, the equality also holds in (\ref{8.2}). This completes the proof of Theorem \ref{main-theo-4}.

\end{document}